\documentclass[12pt]{amsart}  
\usepackage{fullpage, amsmath, amssymb, amsthm, enumerate, url, bibentry, hyperref, color,xcolor, comment}
\usepackage[alphabetic,lite]{amsrefs}
\usepackage{mathrsfs}
\usepackage[all,cmtip]{xy}
\usepackage{multirow}

\usepackage{hyperref}
\hypersetup{
    colorlinks,
    linkcolor={red},
    citecolor={cyan},
    urlcolor={blue}
}

\usepackage[normalem]{ulem}
\usepackage[utf8]{inputenc}
\usepackage{enumitem}

\definecolor{purple}{rgb}{0.59, 0.44, 0.84}

\theoremstyle{plain} 
\newtheorem{thm}{Theorem}[subsection]

\newtheorem{cor}[thm]{Corollary}
\newtheorem{lem}[thm]{Lemma}
\newtheorem{prop}[thm]{Proposition}
\theoremstyle{definition}

\newtheorem{rem}[thm]{Remark}

\DeclareMathOperator{\Kl}{Kl}
\DeclareMathOperator{\Klt}{\widetilde{Kl}}
\DeclareMathOperator{\Ap}{\AA_{\FF_\mathnormal{p}}^1}
\DeclareMathOperator{\Gp}{\mathbb{G}_{\mathnormal{m},\FF_\mathnormal{p}}}
\DeclareMathOperator{\Gptwo}{\mathbb{G}_{\mathnormal{m},\FF_\mathnormal{p}}^2}
\DeclareMathOperator{\Gpk}{\mathbb{G}_{\mathnormal{m},\FF_\mathnormal{p}}^{\mathnormal{k}}}
\DeclareMathOperator{\Gpkone}{\mathbb{G}_{\mathnormal{m},\FF_\mathnormal{p}}^{\mathnormal{k}+1}}
\DeclareMathOperator{\Gpb}{\mathbb{G}_{\mathrm{m},\overline{\FF}_\mathnormal{p}}}
\DeclareMathOperator{\Gpbkone}{\mathbb{G}_{\mathnormal{m},\overline{\FF}_\mathnormal{p}}^{\mathnormal{k}+1}}
\DeclareMathOperator{\AS}{\mathscr{L}_{\psi}}
\DeclareMathOperator{\ASk}{\mathscr{L}_{\mathnormal{f_k}}}
\DeclareMathOperator{\ASkt}{\mathscr{L}_{\mathnormal{\tilde{f}_k}}}
\DeclareMathOperator{\legen}{\mathnormal{\Bigl(\frac{\bullet}{p}\Bigr)}}
\DeclareMathOperator{\et}{\acute{e}t}

\def\BB{\mathbb B}
\def\CC{\mathbb C}
\def\FF{\mathbb F}
\def\PP{\mathbb{P}}
\def\QQ{\mathbb Q}
\def\RR{\mathbb R}
\def\ZZ{\mathbb Z}
\def\({\left(}
\def\){\right)}
\def\ol{\overline}
\def\CCC#1#2{\binom {#1}{#2}}

\newcommand{\bR}{\mathbf{R}}
\renewcommand{\AA}{\mathbb{A}}

\newcommand{\cF}{\mathcal{F}}
\newcommand{\cK}{\mathcal{K}}
\newcommand{\cV}{\mathcal{V}}
\newcommand{\rN}{\mathrm{N}}
\newcommand{\rc}{\mathrm{c}}

\newcommand{\sF}{\mathscr{F}}
\newcommand{\sK}{\mathscr{K}}

\newcommand{\sV}{\mathscr{V}}
\newcommand{\eps}{\varepsilon}
\newcommand{\spm}{\mathsf{m}^+} 
\newcommand{\tm}{\mathsf{m}} 
\newcommand{\spmt}{\widetilde{\mathsf{m}}} 
\newcommand{\qt}{\kappa_2} 
\newcommand{\de}{\mathrm{d}}
\newcommand{\coH}{\mathrm{H}}
\newcommand{\dR}{\mathrm{dR}}
\newcommand{\rmid}{\mathrm{mid}}
\newcommand{\cyc}{\mathrm{cyc}}
\newcommand{\Frob}{\mathrm{Frob}}
\newcommand{\crys}{\mathrm{crys}}
\newcommand{\st}{\mathrm{st}}
\newcommand{\flr}[1]{\left\lfloor #1\right\rfloor}
\newcommand{\Motive}{\mathsf{M}}
\newcommand{\symgp}{\mathfrak{S}}
\newcommand{\FFbar}{\overline{\mathbb{F}}}
\newcommand{\QQbar}{\overline{\mathbb{Q}}}
\newcommand{\cond}{\mathfrak{N}} 
\newcommand{\sVt}{\widetilde{\mathscr{V}}}
\newcommand{\nullt}{{\tilde{0}}} 
\newcommand{\inftyt}{{\widetilde{\infty}}} 
\newcommand{\Gm}{\mathbb{G}_{\mathrm{m}}}
\newcommand{\Gmx}{\mathbb{G}_{\mathrm{m},x}}
\newcommand{\Gmt}{\mathbb{G}_{\mathrm{m},t}}
\newcommand{\Gmz}{\mathbb{G}_{\mathrm{m},z}}
\newcommand{\Aut}{\operatorname{Aut}}
\newcommand{\Gal}{\operatorname{Gal}}
\newcommand{\gr}{\operatorname{gr}}
\newcommand{\rk}{\operatorname{rk}}
\newcommand{\Sym}{\operatorname{Sym}}
\newcommand{\tr}{\operatorname{tr}}
\newcommand{\sw}{\operatorname{sw}} 
\newcommand{\GL}{\operatorname{GL}}
\newcommand{\PGL}{\operatorname{PGL}}
\newcommand{\SO}{\operatorname{SO}}
\newcommand{\SL}{\operatorname{SL}}
\newcommand{\diag}{\operatorname{diag}}
\newcommand{\LMFDB}{\texttt{LMFDB}}
\newcommand{\reven}{\mathrm{even}}
\newcommand{\rodd}{\mathrm{odd}}

\newcommand*\HYPERskip{&}
\catcode`,\active
\newcommand*\pFq{
	\begingroup
	\catcode`\,\active
	\def ,{\HYPERskip}%
	\doHyper
}
\catcode`\,12
\def\doHyper#1#2#3#4#5{%
	\, _{#1}F_{#2}\left[\begin{matrix}#3 \smallskip \\  #4\end{matrix} \; ; \; #5\right]%
	\endgroup
}

\catcode`,\active

\catcode`\,12

\catcode`,\active

\catcode`\,12

\newcommand*\HYPERpp{&}
\catcode`,\active
\newcommand*\pPPq{
	\begingroup
	\catcode`\,\active
	\def ,{\HYPERpp}%
	\doHyperFpp
}
\catcode`\,12
\def\doHyperFpp#1#2#3#4#5{%
	\, _{#1}{\mathbb P}_{#2}\left[\begin{matrix}#3 \smallskip \\  #4\end{matrix} \; ; \; #5\right]%
	\endgroup
}

\title[Twisted Moments of Kloosterman Sums]{Galois Representations and Modularity for Twisted Moments of Kloosterman Sums}
\author[]{Zi Li Lim, Fang-Ting Tu and Jeng-Daw Yu}

\begin{document}

\maketitle

\begin{abstract}
Associated with the classical Kloosterman sums,
we consider the quadratic twisted symmetric power moments.
We investigate the cohomological properties
and construct the motives over the field of the rationals
underlying such moments.
We determine the associated modular forms
in the lower degree cases.
\end{abstract}

\tableofcontents

\section{Introduction}

\subsection{The moments}
Let $p$ be a prime number
and $\FF_p$ the finite field of $p$ elements.
Fix a nontrivial additive character
$\psi\colon \FF_p\to\CC^{\times}$.
Fix an algebraic closure $\FFbar_p$ of $\FF_p$.
For any finite field of $q$ elements $\FF_q$ in $\FFbar_p$,
denote $\tr_{\FF_q/\FF_p}\colon \FF_q\to\FF_p$ the trace map.
For any $a\in \FF_q^{\times}$,
the \textit{Kloosterman sum} $\Kl_2(q;a)$ is the exponential sum given by
\begin{equation}\label{eq:Kl_sum}
    \Kl_2(q;a)=\sum_{x\in\FF_q^{\times}}\psi(\tr_{\FF_q/\FF_p}(x+ax^{-1})).
\end{equation}

It is known that,
for $a\in\FF_q^\times$,
there exist algebraic integers
$\alpha_a, \beta_a$ with $\alpha_a \beta_a=q$
such that
\[ \Kl_2(q^n;a) = -(\alpha_a^n+\beta_a^n),
\quad
n\in\ZZ_{>0}. \]
For a positive integer $k$,
the \textit{symmetric power of degree $k$},
$\Kl_2^{\Sym^k}(q;a)$,
of the Kloosterman sum  is defined as
\begin{gather*}
    \Kl_2^{\Sym^k}(q;a)=\sum_{i=0}^k\alpha_a^{k-i}\beta_a^i.
\end{gather*}
(Notice however the sign change $\Kl_2^{\Sym^1}(q;a) = -\Kl_2(q;a)$
in the $k=1$ case.)
Let $\qt\colon \FF_q^\times\to\CC^\times$
be the nontrivial quadratic character.
The \textit{symmetric power moment}, $\spm_k(q)$,
and the (quadratic) \textit{twisted symmetric power moment of degree $k$},
$\tm_k(q)$,
of the Kloosterman sums are defined respectively as
\[ \spm_k(q) = \sum_{a\in \FF_q^{\times}} \Kl_2^{\Sym^k}(q;a),
\quad
\tm_k(q) =\sum_{a\in \FF_q^{\times}} \qt(a) \Kl_2^{\Sym^k}(q;a). \]
Our main focus in this paper is on the twisted moments,
so we choose to lighten the notation for the twisted case.
It is easy to check that
these moments $\spm_k(q), \tm_k(q)$
are indeed integers.

We also consider a related exponential sum
\[ \Klt_2(q;a) = \sum_{y\in\FF_q^{\times}}\psi(\tr_{\FF_q/\FF_p}a(y+y^{-1})),
\quad
a \in \FF_q^\times, \]
which has a similar decomposition
\[ \Klt_2(q^n;a) = -(\tilde{\alpha}_a^n+\tilde{\beta}_a^n),
\quad
n\in\ZZ_{>0}, \]
with $\tilde{\alpha}_a\tilde{\beta}_a = q$.
(In fact,
one has
$\Klt_2(q;a) = \sum_{x\in\FF_q^{\times}}\psi(\tr_{\FF_q/\FF_p}(x+a^2x^{-1}))
= \Kl_2(q;a^2)$
via the substitution $y=a^{-1}x \in \FF_q^\times$.)
Define accordingly the associated symmetric power moment
\[ \spmt_k(q)
= \sum_{a\in\FF_q^\times}\sum_{i=0}^k\tilde{\alpha}_a^{k-i}\tilde{\beta}_a^i. \]
Since
\[ \spm_k(q) + \tm_k(q) =
2\sum_{b\in (\FF_q^\times)^2} \Kl_2^{\Sym^k}(q;b) \]
and
\[ 2\sum_{b\in (\FF_q^\times)^2}\psi(\tr_{\FF_q/\FF_p}(x+bx^{-1}))^k \\
= \sum_{a\in \FF_q^\times}\psi(\tr_{\FF_q/\FF_p}a(y+y^{-1}))^k \]
by writing $b = (\pm a)^2, x=ay$,
we then obtain
\begin{equation}\label{eq:moments_relation}
\spmt_k(q) = \spm_k(q) + \tm_k(q)
\end{equation}
after turning the symmetric power
$\sum u^{k-i}v^i$ into a combination of powers
$(u+v)^j$ for $0\leq j\leq k$ with $uv = q$.
For instance,
by direct computation, one has
\[ \spm_1(q) = -1,
\quad
\spmt_1(q) = -1-\qt(-1)q,
\quad
\spm_2(q) = -1,
\quad
\spmt_2(q) = -1-q. \]
Thus,
\begin{equation}\label{eq:small_moments}
\tm_1(q) = -\qt(-1)q,
\quad
\tm_2(q) = -q.
\end{equation}
We also have
\[ \spm_3(q) = -1-\qt(-3)q^2,
\quad
\spm_4(q) = -1-q^2, \]
for $p\geq 5$ and $p\geq 3$, respectively.

\subsubsection*{Evans's modularity conjectures}
These are the explicit descriptions
in the next four cases in $k$ for $\spm_k(p)$ and $\tm_k(p)$.
Based on extensive computations,
Ron Evans observes that the sequences of
(the negative of the pure parts of) the moments
$\{\spm_k(p)\}_p$ for $k=5,6,8$
and $\{\tm_k(p)\}_p$ for $k=3,4,6$
coincide, up to constant powers of $p$,
with the Fourier coefficients of certain cuspidal Hecke eigenforms
of weights $k-2$ and $k$, respectively,
at least for $p>k$.
In addition, the moments $\{\spm_7(p)\}_{p>7}$ and $\{\tm_5(p)\}_{p>5}$
are obtained from the symmetric squares
of Hecke eigenforms of weight 3,
twisted by constant powers of $p$ and Dirichlet characters.
See \cite[\S 4]{Evans-hyper}
where our moments $\spm_k(p)$ and $\tm_k(p)$
are denoted by
$T_k(1,0)$ and $T_k(\phi,0)=T_k(\phi,0,1)$, respectively.
See Theorem \ref{thm:EC}
and also Remarks \ref{rem:Evans3} and \ref{rem:Evans5} below
for the precise statements in the case $\tm_k(p)$.
In the appearance of his paper,
this modularity for $\spm_5,\spm_6,\tm_3,\tm_4$
has been obtained by various authors as mentioned therein
(see also the note in \cite[p.530]{Evans-hyper} for $\tm_4$).

\subsubsection*{Galois representations}
Using ideas from the geometric Langlands program
and tools including affine Grassmannians
and homogeneous Fourier transforms,
Yun \cite{Yun} constructs a continuous $\ell$-adic representation of
$\Gal(\QQbar/\QQ)$, for each prime $\ell$,
which is a subquotient of the \'etale cohomology group
of a certain projective smooth variety over $\QQ$,
whose trace of a geometric Frobenius element at an unramified prime $p$
coincides with (the pure part of) the negative moment $-\spm_k(p)$
for each fixed $k$.
Applying Serre's modularity conjecture \cite{Serre-GQ},
proved by Khare and Wintenberger \cite{KW1,KW2},
to such Galois representations,
Yun then proceeds to reprove the modularity
for the moments $\spm_5(p)$ and $\spm_6(p)$,
confirm Evans's modularity conjecture for $\spm_7(p)$,
and reduce that for $\spm_8(p)$
to a finite calculation,
which is done by Vincent in an appendix of the same paper.

On the other hand,
aiming to understand the shape of the completed $L$-function
attached to the pure part of the moments $\spm_k(p)$
for each $k$ and its functional equation,
observed by Broadhurst and Roberts \cite{Broadhurst17}
via intensive numerics,
Fres\'an, Sabbah and Yu \cite{FSY} develop a different approach
for the arithmetic structures of $\spm_k(p)$
originated from a viewpoint of (irregular) Hodge theory.
The authors obtain an explicit hypersurface $\sK$ in a torus
whose \'etale cohomology in the middle degree contains
the Galois representation constructed by Yun,
up to semisimplification,
as a direct summand.
The geometric model is simple enough
to allow the authors to extract detailed information
on ramification behaviors
at bad primes for $\spm_k(p)$,
at least for $p>2$.
Moreover, by an investigation of the characteristic zero analogues
of the Kloosterman sheaf and its symmetric powers,
the Hodge numbers underlying (the motive of) the Galois representation
are determined.
With the current advances of automorphic theory,
this is sufficient to conclude that
the Galois representations are potentially automorphic.

In addition, both approaches \cite{Yun} and \cite{FSY}
apply to the symmetric power moments $\tm_k^{(n+1),+}(p)$
of Kloosterman sums of higher rank $\Kl_{n+1}$
so one obtains the associated Galois representations.
(In fact, Yun constructs such Galois representations
for $V$-moments of the generalized Kloosterman sheaf $\Kl_{\widehat{G}}$
for irreducible representations $V$
of an almost simple split algebraic group $\widehat{G}$ over $\QQ_\ell$.
The moments $\tm_k^{(n+1),+}(p)$ correspond to
$\widehat{G} = \SL_{n+1}$ with
$V$ the $k$-th symmetric power of the standard representation.)
Qin \cite{Qin}
then uses the irregular Hodge approach as in \cite{FSY} to investigate
the Hodge structures underlying $\tm_k^{(n+1),+}(p)$,
and obtains the potential automorphy in certain cases.
Adapting the modularity methods similar to \cite{Yun},
Qin proves the Evans type formulas
relating the moments $\tm_k^{(n+1),+}(p)$ and explicit modular forms
in the rank two cases,
including all the examples listed in \cite[\S1.3]{Yun}
for the group $\widehat{G} = \SL_{n+1}$.

We remark that an explicit \textit{virtual} scheme over $\ZZ$
whose reduction modulo $p$ determines $\tm_k^{(n+1),+}(p)$
is constructed in \cite{FW08}.

\subsection{Main results}
The main goal of this paper is to investigate
the arithmetic
of the sequence of twisted moments
$\{ \tm_k(p) \}$ as $p$ varies.
In particular,
we extend the Galois representation theoretic parts
and the modularity results in \cite{Yun} and \cite{FSY}
to the twisted setting.
We have the following two theorems.

\begin{thm}\label{thm:GR}
For each $k\geq 1$,
let $d = \flr{(k+1)/2} - \delta_{2+4\ZZ}(k)$
where $\delta_{2+4\ZZ}$ is the characteristic function
supported on $2+4\ZZ \subset \ZZ$.
There exists a compatible family of continuous $\ell$-adic representations
for prime numbers $\ell$
\[ \rho_{k,\ell}\colon \Gal(\QQbar/\QQ) \longrightarrow \GL_d(\QQ_\ell), \]
which is unramified and satisfying
\[ \tr\big(\rho_{k,\ell}(\Frob_p)\big) = -\tm_k(p) - \delta_{2+4\ZZ}(k)p^{k/2} \]
at prime $p\neq\ell$
and $p > k$ (resp., $p>k/2$) if $k$ is odd (resp., $k$ is even).
Here $\Frob_p$ is a geometric Frobenius element at $p$
in $\Gal(\QQbar/\QQ)$.
Moreover, $\rho_{k,\ell}$ is self-dual, up to a Tate twist,
is de Rham above $\ell$ and crystalline
if $\ell > \max\{k,2\}$ (resp., $\ell >k/2$) when $k$ is odd (resp., $k$ is even),
and is potentially automorphic.
\end{thm}

In the following,
the {\LMFDB}
is referred to the online database \cite{LMFDB}
of $L$-functions and modular forms.

\begin{thm}\label{thm:EC}
For a cuspidal Hecke eigenform $f$ on some congruence subgroup,
let
$f(q)=\sum_{n=1}^\infty a_f(n)q^n$
be its $q$-expansion.
The following modularity holds.
\begin{enumerate}
\item
Let $f$ be the modular form $f_{\text{48.3.e.a}}$ (CM by $\QQ(\sqrt{-3}))$
in {\LMFDB}.
Then, for $p\geq 5$,
\begin{equation}\label{eq:Evans3}
\tm_3(p) = -a_f(p)\cdot p.
\end{equation}
\item
Let $f$ be the modular form $f_{\text{8.4.a.a}}$ in {\LMFDB}.
Then, for each odd prime $p$,
\[ \tm_4(p) = -a_f(p)\cdot p \]
\item
Let $f$ be the modular form $f_{\text{60.3.b.a}}$ in {\LMFDB}
(with $a_f(n)\in \QQ(\sqrt 5, i)$).
Then, when $p\geq 7$,
\begin{equation}\label{eq:Evans5}
\tm_5(p) =
- \Bigl(\frac{-1}{p}\Bigr)a_f(p)^2\cdot p + \Bigl(\frac{15}p\Bigr)p^3,
\end{equation}
where $\(\frac{\cdot}p\)$ is the Legendre symbol. 
\item
Let $f$ be the modular form $f_{\text{24.6.a.a}}$ in {\LMFDB}.
Then, when $p\geq 5$,
\[ \tm_6(p) = -a_f(p)\cdot p-p^3. \]
\end{enumerate}
\end{thm}

This confirms Evans's modularity conjectures for $\tm_k(p)$.

\medskip
Theorem \ref{thm:GR}
is a simplified version of the contents of
Propositions \ref{oddrep}, \ref{evenrep} and \ref{prop:p-adicHodge},
and Corollary \ref{cor:pot_auto} below.
In fact, such Galois representations are already constructed
in \cite{FSY}, at least implicitly,
as a direct summand of the \'etale cohomology
of the same hypersurface $\sK$ for $\spm_k(p)$ mentioned previously.
Essentially, the only missing point is to pin down the determinants
of the representations $\rho_{k,\ell}$.
The latter relies on the results obtained in \S \ref{sect:twistedFW},
in which we investigate in detail
the \'etale cohomology associated with $\tm_k(p)$
for fixed $k$ and $p$.
The consideration extends the work \cite{FW10}
for the untwisted case $\spm_k(p)$.
The modularity Theorem \ref{thm:EC} is obtained using similar arguments
as in \cite{Yun}.
\medskip

The rest of the paper is organized as follows.
In \S \ref{sect:twistedFW},
we recall the construction of the Kloosterman sheaf $\Kl_2$ on $\Gm$
in positive characteristic.
We investigate in detail the local structures of the sheaf $\Klt_2$,
which is the pullback of $\Kl_2$ by a double cover of $\Gm$,
and use them to obtain the cohomological properties
for $\tm_k(p)$ for fixed positive $k$ and odd $p$.
The approach and the main results in this section
are based on the work \cite{FW10} in an essential way.
In \S \ref{sect:twistedSFY},
we construct the associated Galois representation for $\tm_k(p)$
for fixed $k$, mainly by extracting the necessary information
from \cite{FSY}, and draw some arithmetic and analytic consequences
from the potential automorphy of the representation.
We then prove Theorem \ref{thm:EC}
relating $\tm_k(p)$ for $k=3,4,5,6$ to modular forms
using Serre's modularity conjecture.

In Appendix \S\ref{sect:Lim},
we indicate that one can obtain the main results in \S \ref{sect:twistedFW}
by direct inspections of the quadratic twists of the symmetric powers
of $\Kl_2$,
instead of going up to $\Klt_2$.
Still, this approach relies on the calculations in \cite{FW10}.
In \S\ref{sect:HCS},
we record the relationship between the twisted Kloosterman sums and hypergeometric type character sums and their modularity.
We record some facts about the associated $L$-functions
in \S\ref{sect:A.L-func}.

\subsubsection*{Acknowledgements} Tu is partially supported by the NSF grant DMS \#2302531. While working on this project, Tu was hosted by the Department of Mathematics at National Taiwan University and the Taida Institute for Mathematical Sciences in the summers of 2024 and 2025.
Yu is grateful for the partial support of the grant NTU-CC-114L891801,
which made the collaboration feasible.

\section{Cohomology for the moments}
\label{sect:twistedFW}

Fix an odd prime $p$ and a positive integer $k$.
We investigate in detail the twisted moments $\tm_k(p)$ cohomologically
in this section.

\subsection{Étale sheaves}
We work over the base field $\FF_p$.
When we need to specify a variable for the torus over $\FF_p$,
we set $\Gmx = \operatorname{Spec}\FF_p[x,x^{-1}]$,
indicating the variable $x$ in the subscript;
otherwise we omit it.
A similar convention applies to the affine line $\AA^1$
and the projective line $\PP^1$.
We recall briefly the construction of the Kloosterman sheaf
and some basic properties that we take for granted.

We fix a rational prime $\ell$, distinct from $p$,
and a nontrivial additive character
\[ \psi\colon \FF_p \longrightarrow \QQbar_\ell^\times. \]
(One can replace the coefficient field $\QQbar_\ell$
by the extension of $\QQ_\ell$
joined with $p$-th roots of 1.)
We extend $\psi$ to a character of any finite field extension of $\FF_p$
by composing with the trace map.
The Artin--Schreier cover
$\AA^1_t\to\AA^1_s$ induced from $s\mapsto t^p-t$
is a Galois cover;
it then corresponds to a quotient
$\pi_1^{\et}(\AA^1_s) \to \Aut(\AA^1_t/\AA^1_s) = \FF_p$
of the \'etale fundamental group of $\AA^1$.
Composing with $\psi^{-1}$
to obtain a character of $\pi_1^{\et}(\AA^1_s)$,
we get the \textit{Artin--Schreier sheaf} $\AS$,
which is a lisse sheaf on $\AA^1$
with Swan conductor $\sw_\infty(\AS) = 1$ at $\infty$.
It is also characterized by the following trace function.
For a closed point $x\in\AA^1(\FF_q) = \FF_q$,
let $\Frob_x$ be a geometric Frobenius at $x$.
Then $\Frob_x|_{\AS} = \psi(x)$.
For a regular function $f: X \to \AA^1$ of $X$ over $\FF_p$,
let $\AS(f)$ be the pull-back $f^*\AS$.
For a geometric Frobenius $\Frob_x$ at $x\in X(\FF_q)$,
we then have $\Frob_x|_{\AS(f)} = \psi(f(x))$.
One has the duality
\begin{equation}\label{eq:AS_duality}
\AS(f)^\vee = \AS(-f).
\end{equation}

Consider the diagram
\[ \Gmz \overset{\pi}{\longleftarrow} \Gmz\times\Gmx
\overset{f}{\longrightarrow} \AA^1_s \]
where $\pi$ is the projection to the first component and $f$ is induced by
$s\mapsto x+zx^{-1}$.
The \textit{Kloosterman sheaf} is the lisse sheaf $\Kl_2$ of rank two
on $\Gm$ over $\FF_p$
defined as
\[ \Kl_2 = \bR^1\pi_!\AS(f) = \bR^1\pi_*\AS(f). \]
It is pure of weight one.
For a closed point $a\in \Gm(\FF_q) = \FF_q^\times$, one has
$\tr(\Frob_a\mid\Kl_2) = -\Kl_2(q;a)$,
the negative of the Kloosterman sum valued in $\QQbar_\ell$,
instead of $\CC$,
as defined in \eqref{eq:Kl_sum}.
Since the automorphism
$(z,x) \mapsto (-z,-x)$ on $\Gmz\times\Gmx$
sends $f$ to $-f$,
the duality \eqref{eq:AS_duality}
induces a perfect alternating pairing
\begin{equation}\label{eq:Kl_duality}
\Kl_2 \otimes \Kl_2 \longrightarrow \bR^2\pi_!\QQbar_\ell = \QQbar_\ell(-1).
\end{equation}

On the other hand,
let
\[ [2]\colon \Gmt \longrightarrow \Gmz \]
be the double cover induced by $z\mapsto t^2$.
It yields a surjection
$\pi_1^{\et}(\Gmz) \to \Aut(\Gmt/\Gmz) = \{\pm 1\}$.
Composing with the nontrivial character
$\{\pm 1\} \subset \QQbar_\ell^\times$,
we obtain a lisse sheaf $\qt$ on $\Gm$.
One has
\[ [2]_*\QQbar_\ell = \QQbar_\ell\oplus\qt. \]

Let $\Klt_2 = [2]^*\Kl_2$
and
\[ \sVt_k = \Sym^k\Klt_2,
\quad
\sV^+_k = \Sym^k\Kl_2,
\quad
\sV_k = \qt\otimes\Sym^k\Kl_2, \]
the $k$-th symmstric power of $\Klt_2$ and $\Kl_2$
and its quadratic twist, respectively.
So for $a\in\Gm(\FF_q) = \FF_q^\times$, one has
$\tr(\Frob_a\mid \sV^+_k) = \Kl_2^{\Sym^k}(q;a)$,
the symmetric power of $\Kl_2(q;a)$,
and similarly for $\sVt_k$ and $\sV_k$.
There are $(-1)^k$-symmetric perfect pairings
\begin{equation}\label{eq:duality_sheaf}
\sF\otimes\sF \longrightarrow \QQbar_\ell(-k),
\quad
\sF = \sVt_k, \sV^+_k, \sV_k
\end{equation}
inherited from \eqref{eq:Kl_duality}
and the self-duality of $\qt$.
We have
\begin{equation}\label{eq:sheafV_relation}
[2]^*\sV^+_k = \sVt_k,
\quad
[2]_*\sVt_k = \sV^+_k \oplus \sV_k;
\end{equation}
the letter is the sheaf avatar of \eqref{eq:moments_relation}.
By the Lefschetz trace formula,
the twisted moment $\tm_k(p)$
is determined by the Frobenius action
on the \'etale cohomology with compact support
$\coH_{\et,\rc}^i(\Gpb,\sV_k)$.
For $\sF = \sVt_k, \sV^+_k, \sV_k$,
there is the duality pairing
\begin{equation}\label{eq:cempty_duality}
\coH_{\et,\rc}^i(\Gpb,\sF) \otimes \coH_{\et}^{2-i}(\Gpb,\sF)
\longrightarrow \coH_{\et,\rc}^2(\Gpb, \QQbar_\ell(-k)) = \QQbar_\ell(-k-1)
\end{equation}
induced by \eqref{eq:duality_sheaf}.
We also consider the \textit{middle part} of the cohomology
\[ \coH_{\et,\rmid}^i(\Gpb,\sF)
= \operatorname{image}\big\{
\coH_{\et,\rc}^i(\Gpb,\sF) \longrightarrow \coH_{\et}^i(\Gpb,\sF) \big\} \]
where the arrow is the forgetting-support map.
It inherits the $(-1)^{k+1}$-symmetric perfect pairing
\begin{equation}\label{eq:mid_duality}
\coH_{\et,\rmid}^1(\Gpb,\sF) \otimes \coH_{\et,\rmid}^1(\Gpb,\sF)
\longrightarrow \QQbar_\ell(-k-1).
\end{equation}
In fact, letting $j\colon \Gm \to \PP^1$ be the inclusion,
one has the identity
\begin{equation}\label{eq:mid_j*}
\coH_{\et,\rmid}^1(\Gpb,\sF)
= \coH_{\et}^1(\PP^1_{\FFbar_p},j_*\sF),
\quad
\sF = \sVt_k, \sV^+_k, \sV_k.
\end{equation}
As $\Gal(\FFbar_p/\FF_p)$-modules,
these spaces are pure of weight $k+1$.

We record the cohomological properties for $\sV^+_k$
obtained in \cite{FW05, FW10} in the following proposition.
Let $I_0$ and $I_\infty$ be the inertia groups of $\pi_1^{\et}(\Gm)$
at the boundary points $0,\infty\in\PP^1\setminus\Gm$, respectively,
acting on $\sVt_k, \sV^+_k, \sV_k$.
Let $F_p\in\Gal(\FFbar_p/\FF_p)$
be the geometric Frobenius automorphism.

\begin{prop}\label{prop:FW}
As $\Gal(\FFbar_p/\FF_p)$-modules,
there is the exact sequence
\[ 0 \longrightarrow (\sV_k^+)^{I_0}\oplus (\sV_k^+)^{I_\infty}
\longrightarrow \coH_{\et,\rc}^1(\Gpb,\sV^+_k)
\longrightarrow \coH_{\et,\rmid}^1(\Gpb,\sV^+_k)
\longrightarrow 0. \]
The space $(\sV_k^+)^{I_0}$ is of dimension one,
upon which $\Gal(\FFbar_p/\FF_p)$ acts trivially.
The following properties hold.
\begin{enumerate}
\item
Suppose $k$ is odd.
Then $(\sV_k^+)^{I_\infty} = 0$
and
\begin{equation}\label{eq:FWdet_odd}
\det\big(F_p\mid\coH_{\et,\rmid}^1(\Gpb,\sV^+_k)\big)
= p^{\frac{k+1}{2}\delta^+}
\left(\frac{-2}{p}\right)^{\flr{\frac{k}{2p}+\frac{1}{2}}}
\prod_{\substack{0\leq j\leq (k-1)/2\\ p\nmid (2j+1)}} \left(\frac{(-1)^j(2j+1)}{p}\right)
\end{equation}
where
\[ \delta^+ = \dim \coH_{\et,\rmid}^1(\Gpb,\sV^+_k)
= \frac{k-1}{2}-\flr{\frac{k}{2p}+\frac{1}{2}}. \]
\item
Suppose $k$ is even.
\begin{enumerate}
\item
$\Gal(\FFbar_p/\FF_p)$ acts on $(\sV_k^+)^{I_\infty}$
semisimply and
\[ \det(1-F_pT \mid (\sV_k^+)^{I_\infty})
= (1-p^{k/2}T)^{\delta_{4\ZZ}(k)+n_k^+(p)}
(1-(-1)^{(p-1)/2}p^{k/2}T)^{n_k^-(p)}
\]
with
\begin{equation}\label{eq:n_k}
n_k^+(p) = \flr{\frac{k}{4p}},
\quad
n_k^-(p) = \flr{\frac{k}{4p}+\frac{1}{2}}.
\end{equation}
\item
We have
$\dim \coH_{\et,\rc}^1(\Gpb,\sV^+_k) = \frac{k}{2}-\flr{\frac{k}{2p}}$.
\item
One has
\[ \det\big(F_p\mid\coH_{\et,\rmid}^1(\Gpb,\sV^+_k)\big)
= p^{\frac{k+1}{2}\delta^+} \]
where
\[ \delta^+ = \dim \coH_{\et,\rmid}^1(\Gpb,\sV^+_k)
= \frac{k}{2}-2\flr{\frac{k}{2p}} - \delta_{4\ZZ}(k). \]
\end{enumerate}
\end{enumerate}
\end{prop}

The rest of this section is devoted to proving the following theorem.
By the decomposition in \eqref{eq:sheafV_relation},
we have
\begin{equation}\label{eq:cohV_relation}
\coH_{\et,?}^i(\Gpb,\sVt_k) =
\coH_{\et,?}^i(\Gpb,\sV^+_k) \oplus \coH_{\et,?}^i(\Gpb,\sV_k),
\quad
? = \emptyset,\rc,\rmid.
\end{equation}
What we do below is indeed to obtain the corresponding results for $\sVt_k$
based on the work of Fu and Wan,
thanks to the equalities in \eqref{eq:sheafV_relation}.

\begin{thm}\label{thm:zeta_at_p}
\begin{enumerate}
\item
Suppose $k$ is odd.
We have the canonical isomorphisms
\[ \coH_{\et,\rc}^1(\Gpb,\sV_k) \overset{\sim}{\longrightarrow}
\coH_{\et,\rmid}^1(\Gpb,\sV_k) \overset{\sim}{\longrightarrow}
\coH_{\et}^1(\Gpb,\sV_k). \]
These spaces are of dimension
$\delta=\frac{k+1}{2}-\flr{\frac{k}{2p}+\frac{1}{2}}$.
Moreover, one has
\begin{equation}\label{eq:det_odd}
\det\big(F_p\mid\coH_{\et,\rmid}^1(\Gpb,\sV_k)\big)
= p^{\frac{k+1}{2}\delta}
\left(\frac{-1}{p}\right)^\delta
\left(\frac{2}{p}\right)^{\flr{\frac{k}{2p}+\frac{1}{2}}}
\prod_{\substack{0\leq j\leq (k-1)/2\\ p\nmid (2j+1)}} \left(\frac{(-1)^j(2j+1)}{p}\right).
\end{equation}
\item
Suppose $k$ is even.
As $\Gal(\FFbar_p/\FF_p)$-modules,
there is the exact sequence
\[ 0 \longrightarrow \sV_k^{I_\infty}
\longrightarrow \coH_{\et,\rc}^1(\Gpb,\sV_k)
\longrightarrow \coH_{\et,\rmid}^1(\Gpb,\sV_k)
\longrightarrow 0. \]
The following properties hold.
\begin{enumerate}
\item
$\Gal(\FFbar_p/\FF_p)$ acts on $\sV_k^{I_\infty}$
semisimply and
\[ \det(1-F_pT \mid \sV_k^{I_\infty})
= (1-p^{k/2}T)^{\delta_{2+4\ZZ}(k)+n_k^+(p)}
(1-(-1)^{(p-1)/2}p^{k/2}T)^{n_k^-(p)}
\]
where $n_k^+(p), n_k^-(p)$ are defined in \eqref{eq:n_k}.
\item
We have
$\dim \coH_{\et,\rc}^1(\Gpb,\sV_k) = \frac{k}{2}-\flr{\frac{k}{2p}}$.
\item
One has
\[ \det\big(F_p\mid\coH_{\et,\rmid}^1(\Gpb,\sV_k)\big)
= p^{\frac{k+1}{2}\delta} \]
where
\[ \delta = \dim \coH_{\et,\rmid}^1(\Gpb,\sV_k)
= \frac{k}{2}-2\flr{\frac{k}{2p}} - \delta_{2+4\ZZ}(k). \]
\end{enumerate}
\end{enumerate}
\end{thm}

Let
\begin{align}
\label{eq:localfactor_Z}
Z_k(p;T)
&= \det\big(1-F_pT\mid\coH_{\et,\rc}^1(\Gpb,\sV_k)\big), \\
\nonumber
M_k(p;T)
&= \det\big(1-F_pT\mid\coH_{\et,\rmid}^1(\Gpb,\sV_k)\big).
\end{align}
By Lemma \ref{lem:coh_vanish} below,
one has
$\tm_k(p) = Z'_k(p;0)$.
Due to the self-duality \eqref{eq:mid_duality},
we have the functional equation
\[ M_k(p;T) = c T^\delta M_k(p;(p^{k+1}T)^{-1}) \]
where
$\delta = \deg M_k(p;T)$, and
\begin{equation}\label{eq:c=det}
c =
\det\big(-F_p\mid\coH_{\et,\rmid}^1(\Gpb,\sV_k)\big)
= (-1)^\delta\det\big( F_p\mid\coH_{\et,\rmid}^1(\Gpb,\sV_k)\big).
\end{equation}

\subsection{Local structures}
We investigate the local structures of $\sVt_k$
by pulling back the results obtained in \cite{FW10}
via \eqref{eq:sheafV_relation}.

Let $j\colon \AA^1 \to \PP^1$ be the smooth compactification over $\FF_p$.
For a closed point $x\in|\PP^1|$,
we let $\PP^1_{(x)}$ be the henselization at $x$,
and let $\eta_x$ be its generic point.
Choose a geometric point $\bar{\eta}_x$ above $\eta_x$.
Let $I_x$ be the inertia subgroup of the Galois group
$\Gal(\bar{\eta}_x/\eta_x)$.

At the point $0$,
let $t_\ell\colon I_0 \to \QQ_\ell(1) = \varprojlim \mu_{\ell^n}(\QQbar_\ell)$
be the $\ell$-adic fundamental tame character given by
\[ t_\ell(\sigma) = \varprojlim \frac{\sigma(\sqrt[\ell^n]{z})}{\sqrt[\ell^n]{z}}. \]
Here, $\sqrt[\ell^n]{z}$ is any choice of $\ell^n$-th root of $z$
on $\bar{\eta}_0$.
Fix a lifting $\widetilde{F}_p \in \Gal(\bar{\eta}_0/\eta_0)$
of the geometric Frobenius $F_p\in\Gal(\FFbar_p/\FF_p)$.

To distinguish the considerations of the sheaves $\sV^+_k$ and $\sVt_k$,
we let $\eta_\nullt$ be the corresponding point above $\eta_0$
under the double cover $\PP^1_t \to \PP^1_z$
with a geometric point $\bar{\eta}_\nullt$ above $\eta_\nullt$.
We have the inertia subgroup $I_\nullt$ and an element $\widetilde{F}_p$
of $\Gal(\bar{\eta}_\nullt/\eta_\nullt)$
mapping to the geometric Frobenius.

\begin{lem}\label{lem:sVt_0}
There exists a basis $\{\tilde{f}_i\}_{i=0}^k$
of $\sVt_k|_{\eta_\nullt}$
such that
$\widetilde{F}_p(\tilde{f}_i) = p^i\tilde{f}_i$
and, for $\sigma\in I_\nullt$,
\[ \sigma(\tilde{f}_i) = \exp(t_\ell(\sigma)\rN)\tilde{f}_i,
\quad
\text{with $\rN(\tilde{f}_i) = \tilde{f}_{i-1}$}. \]
Here one puts $\tilde{f}_{-1} = 0$.
In particular,
$\sVt_k|_{\eta_\nullt}$
is tamely ramified,
$\sVt_k^{I_\nullt} = \langle \tilde{f}_0\rangle$,
and $F_p|_{\sVt_k^{I_\nullt}} = 1$.
\end{lem}

\begin{proof}
By \cite[Lemma 1.2]{FW10},
there exists a basis $\{f_i\}_{i=0}^k$ of the restriction $\sV^+_k|_{\eta_0}$
such that
$\widetilde{F}_p(f_i) = p^if_i$
and, for $\sigma\in I_0$,
\[ \sigma(f_i) = \exp(t_\ell(\sigma)\rN)f_i,
\quad
\text{with $\rN(f_i) = f_{i-1}$} \]
where we set $f_{-1}=0$.
By pulling back to $\eta_\nullt$,
we obtain the statements.
\end{proof}

We turn to the infinity point.
We again distinguish the infinity points $\infty$ and $\inftyt$
for $\sV^+_k$ and $\sVt_k$, respectively.

Consider the quadratic Gauss sum
\[ g = -\sum_{x\in\FF_p^\times}\Big(\frac{x}{p}\Big)\psi(x). \]
It is well-known that
\begin{equation}\label{eq:gauss_square}
g^2 = (-1)^{(p-1)/2}p.
\end{equation}
We have the natural surjection
$\pi_1^{\et}(\Gm) \to \Gal(\FFbar_p/\FF_p)$
whose kernel is the geometric fundamental group.
Introduce two characters
\[ \theta_0, \theta_1\colon
\pi_1^{\et}(\Gm) \longrightarrow \QQbar_\ell^\times \]
given by composing characters
$\theta'_i\colon \Gal(\FFbar_p/\FF_p) \to \QQbar_\ell^\times$, respectively,
where
\[ \theta'_0(F_p) = g,
\quad
\theta'_1(F_p) = (-1)^{(p-1)/2}. \]

We set
\[ \cK_i^+ = \AS(-2(k-2i)t),
\quad
\cK_i = \AS(-2(k-2i)t)\otimes\qt,
\quad
\cF_i = \cF(\theta_0^k\theta_1^i). \]
Here, for an integer $a$,
$\AS(at)$ is the pullback of $\AS$ on $\AA^1_t$
by the map induced by $t\mapsto at$,
and $\cF(\theta_0^a\theta_1^b)$ is the lisse sheaf
corresponding to the character $\theta_0^a\theta_1^b$.
Notice that the sheaf $\cF_i$ is unramified
and $\qt$ is tamely ramified at $\PP^1\setminus\Gm$.
For $a\in\ZZ$,
the sheaf $\AS(at)$, lisse on $\AA^1$,
is ramified at $\inftyt$
if and only if $p\nmid a$.
In the latter case, it has Swan conductor
$\sw_\inftyt(\AS(at)) = 1$.

Let
\begin{equation}\label{eq:Xi}
\Xi_k(p)
= \big\{ i\in\ZZ_{\geq 0} \mid 0\leq i\leq \flr{(k-1)/2}, k-2i\equiv 0 \bmod{p} \big\}.
\end{equation}

\begin{lem}\label{lem:sVt_infty}
We have the decomposition of
$\sVt_k|_{\eta_\inftyt}$
as follows.
\begin{enumerate}
\item
If $k=2r+1$ is odd,
\begin{equation}\label{eq:decompKF}
\sVt_k|_{\eta_{\inftyt}}
= \bigoplus_{i=0}^r \cV_i^{\oplus 2},
\quad
\cV_i = \cK_i\otimes\cF_{i+1}
\end{equation}
with Swan conductor
$\sw(\cV_i) = 1-\delta_{\Xi_k(p)}(i)$.
In particular,
$\sVt_k^{I_\inftyt} = 0$.
\item
In case $k=2r$ is even,
\[ \sVt_k|_{\eta_{\inftyt}}
= \cF_r
\oplus \bigoplus_{i=0}^{r-1} (\cV_i^+)^{\oplus 2},
\quad
\cV_i^+ = \cK_i^+\otimes\cF_i \]
with Swan conductor
$\sw(\cV_i^+) = 1-\delta_{\Xi_k(p)}(i)$.
Moreover,
$\Gal(\FFbar_p/\FF_p)$ acts on $\sVt_k^{I_\inftyt}$
semisimply and
\begin{equation}\label{eq:ch_poly_inftyt}
\det(1-F_pT \mid \sVt_k^{I_\inftyt})
= (1-p^{k/2}T)^{1+2n_k^+(p)}(1-(-1)^{(p-1)/2}p^{k/2}T)^{2n_k^-(p)}
\end{equation}
with $n_k^{\pm}(p)$ defined previously in \eqref{eq:n_k}.
\end{enumerate}
\end{lem}

\begin{proof}
Recall \cite[Lemma 2.4]{FW10} that
\[ {\sV_k^+}|_{\eta_{\infty}}
= \begin{cases}\displaystyle
\bigoplus_{i=0}^r [2]_*\cK_i
	\otimes\cF_{i+1}
& \text{for $k=2r+1$ odd}, \\
\displaystyle
\qt^r\otimes\cF_r
\oplus \bigoplus_{i=0}^{r-1} [2]_*\cK_i^+
	\otimes\cF_i
& \text{for $k=2r$ even}.
\end{cases} \]
Pulling back to $\eta_\inftyt$,
the statements follow,
except for the last equality \eqref{eq:ch_poly_inftyt}.
For $i \in \Xi_k(p)$,
$\cK_i^+$ is trivial and $\cV_i^+ = \cF_i$,
so one has
\[ \sVt_k^{I_\inftyt}
= \cF_r
\oplus \bigoplus_{i\in\Xi_k(p)} \cF_i^{\oplus 2}. \]
By \eqref{eq:gauss_square},
we have $F_p|_{\cF_r} = (\theta_0^2\theta_1)^r(F_p) = p^r$ and,
for $i\in\Xi_k(p)$,
\begin{align*}
F_p|_{\cF_i} = (\theta_0^k\theta_1^i)(F_p)
&= \big((-1)^{(p-1)/2}\big)^{r-i}p^r \\
&= \begin{cases}
p^r, & (k-2i)/p \equiv 0 \bmod{4}, \\
(-1)^{(p-1)/2}p^r, & (k-2i)/p \equiv 2 \bmod{4}. \\
\end{cases}
\end{align*}
Noticing the counting
\begin{align*}
\#\{ i\in\Xi_k(p) \mid (k-2i)/p \equiv 0\bmod{4} \}
&= \flr{\frac{k}{4p}} = n_k^+(p), \\
\#\{ i\in\Xi_k(p) \mid (k-2i)/p \equiv 2\bmod{4} \}
&= \flr{\frac{k}{4p}+\frac{1}{2}} = n_k^-(p),
\end{align*}
the formula \eqref{eq:ch_poly_inftyt} then follows.
\end{proof}

\subsection{The cohomology}

\begin{lem}\label{lem:coh_vanish}
We have the vanishing
$\coH^i_{\et,?}(\Gpb,\sVt_k)=0$
for $i\neq 1$ and $?=\emptyset, \rc, \rmid$.
\end{lem}

\begin{proof}
Since $\Gm$ is an affine curve,
we have $\coH^i_{\et}(\Gpb,\sVt_k)=0$ for $i\neq 0,1$
by Artin vanishing.
By \cite[Thm.11.1]{Katz88},
the geometric monodromy group of $\Kl_2$
equals $\mathrm{SL}_2$ over $\QQbar_\ell$, which is connected.
Therefore, the geometric monodromy group
of the pullback $\Klt_2$ by the double cover $[2]$
remains equal to $\mathrm{SL}_2$.
Since the symmetric power $\sVt_k$
of the standard representation $\sVt_1$
of $\mathrm{SL}_2$ has no nontrivial invariant,
one obtains $\coH^0_{\et}(\Gpb,\sVt_k)=0$.
The duality pairing \eqref{eq:cempty_duality} for $\sF = \sVt_k$
then yields the vanishing $\coH^i_{\et,\rc}(\Gpb,\sV_k) = 0$ for $i\neq 1$.
\end{proof}

Consequently, one has
\[ \dim\coH_{\et}^1(\Gpb,\sVt_k)
= \dim\coH_{\et,\rc}^1(\Gpb,\sVt_k)
= -\chi(\PP^1, j_!\sVt_k). \]
By the Grothendieck--Ogg--Shafarevich formula
and Lemmas \ref{lem:sVt_0},  \ref{lem:sVt_infty},
one obtains readily
\begin{align*}
\chi(\PP^1, j_!\sVt_k)
= -\sw_\nullt(\sVt_k) - \sw_\inftyt(\sVt_k)
&= -2\big(\flr{(k+1)/2}-|\Xi_k(p)|\big) \\
&= \begin{cases}
-2\big(\frac{k+1}{2}-\flr{\frac{k}{2p}+\frac{1}{2}}\big), & \text{$k$ odd}, \\
-2\big(\frac{k}{2}-\flr{\frac{k}{2p}}\big), & \text{$k$ even}.
\end{cases}
\end{align*}

On the other hand, consider the exact sequence on $\PP^1$
\[ 0 \longrightarrow j_!\sVt_k
\longrightarrow j_*\sVt_k
\longrightarrow \sVt_k^{I_\nullt}\oplus\sVt_k^{I_\inftyt}
\longrightarrow 0. \]
Taking $\bR\Gamma$
and by \eqref{eq:mid_j*} and Lemma \ref{lem:coh_vanish},
it yields
\[ 0 \longrightarrow \sVt_k^{I_\nullt}\oplus\sVt_k^{I_\inftyt}
\longrightarrow \coH_{\et,\rc}^1(\Gpb,\sVt_k)
\longrightarrow \coH_{\et,\rmid}^1(\Gpb,\sVt_k)
\longrightarrow 0. \]

Up to this point,
we obtain the statements in Theorem \ref{thm:zeta_at_p}
except the formulas for
$\det(F_p\mid \coH^1_{\et,\rmid}(\Gpb, \sV_k))$.

\subsubsection*{The determinant}
Due to the existence of $(-1)^{k+1}$-symmetric perfect pairing
\eqref{eq:mid_duality}
on the pure $\Gal(\FFbar_p/\FF_p)$-module
$\coH^1_{\et,\rmid}(\Gpb, \sV_k)$,
we have
\[ \det\big(F_p\mid \coH^1_{\et,\rmid}(\Gpb, \sV_k)\big)
= \pm p^{\delta(k+1)/2} \]
with positive sign for $k$ even.
It remains to determine the sign in the case where $k$ is odd.

Together with \eqref{eq:c=det}, we let
\[ c^+ = \det(-F_p\mid\coH_{\et}^1(\PP^1_{\FFbar_p},j_*\sV_k^+)),
\quad
\tilde{c} = \det(-F_p\mid\coH_{\et}^1(\PP^1_{\FFbar_p},j_*\sVt_k)). \]
By \eqref{eq:cohV_relation}, we have the relation
\[ \tilde{c} = c^+\cdot c. \]
The value $c^+$ is given by $(-1)^{\delta^+}\cdot\eqref{eq:FWdet_odd}$
for $k$ odd.
So we are reduced to computing $\tilde{c}$.

By the product formula \cite[Th.(3.2.1.1)]{Laumon},
\[ \tilde{c} = p^{k+1}\prod_{x\in |\PP^1|}
\eps(\PP^1_{(x)}, {j_*{\sVt_k}}_{(0)},\omega_{(x)}) \]
for any nonzero global meromorphism $1$-form $\omega$ on $\PP^1$.
Each term in the product in the equality
is the local $\eps$-factor, see \cite[Th.(3.1.5.4)]{Laumon}.
We now take
$\omega = \de t^2 = [2]^*\de z$.
Since, on $\Gm$,
$\omega$ has no pole nor zero
and $j_*{\sVt_k}$ is lisse,
$\eps(\PP^1_{(x)}, {j_*{\sVt_k}}_{(0)},\omega_{(x)}) = 1$
for $x\in |\Gm|$.
One reduces
\begin{equation}\label{eq:eps_0_infinity}
\tilde{c} =
p^{k+1}\eps(\PP^1_{(\nullt)}, {j_*{\sVt_k}}_{(\nullt)},\omega_{(\nullt)})\cdot
\eps(\PP^1_{(\inftyt)}, {j_*{\sVt_k}}_{(\inftyt)},\omega_{(\inftyt)}).
\end{equation}

By \cite[Lemma 1.2]{FW10} and Lemma \ref{lem:sVt_0},
we have
${j_*{\sVt_k}}_{(\nullt)} \simeq {j_*\sV_k^+}_{(0)}$.
Therefore we obtain
\begin{align}
\nonumber
\eps(\PP^1_{(\nullt)}, {j_*{\sVt_k}}_{(\nullt)}, \omega_{(\nullt)})
&= \eps(\PP^1_{(0)}, {j_*\sV_k^+}_{(0)}, {\de z^2}_{(0)}) \\
\nonumber
&= p^{k(k+1)/2}p^{k+1}\eps(\PP^1_{(0)}, {j_*\sV_k^+}_{(0)}, {\de z}_{(0)})
\quad
\text{(\cite[(3.1.5.5)]{Laumon})} \\
\label{eq:eps_0}
&= -p^{(k+1)^2}
\quad
\text{(\cite[Prop.1.4]{FW10})}.
\end{align}
Here, in the second equality,
one uses that
$\de z^2 = 2z\de z$ has a simple zero at $0$,
$\det{j_*\sV_k^+}_{(0)} = \QQbar_\ell(-\frac{k(k+1)}{2})$
and the corresponding reciprocity map
sends $2z$ to $p^{k(k+1)/2}$.

Recall the local structure in Lemma \ref{lem:sVt_infty}
of $\sVt_k$ at $\inftyt$.
We have, for $\alpha\in\FF_p$,
the evaluations of the local conductor (\cite[(3.1.5.1)]{Laumon})
\[ a(\alpha) :=
a(\PP^1_{(\infty)}, {j_*\AS(\alpha t)\otimes\qt}_{(\infty)}, \omega_{(\infty)})
= \begin{cases}
-2, & \alpha = 0, \\
-1, & \alpha \neq 0.
\end{cases}\]
One has, for $\cV_i = \cK_i\otimes\cF_{i+1}$,
\begin{align*}
\eps(\PP^1_{(\infty)}, {j_*\cV_i}_{(\infty)}, \omega_{(\infty)})
&= \det(F_p\mid\cF_{i+1})^{a(4i-2k)}
	\eps(\PP^1_{(\infty)}, {j_*\cK_i}_{(\infty)}, \omega_{(\infty)})
	\quad
	\text{(\cite[(3.1.5.6)]{Laumon})} \\
&= \begin{cases}
\pm g^{-2k}\cdot p^{-3}g, & p\mid (k-2i) \\
\pm g^{-k}\cdot p^{-2}, & p\nmid (k-2i)
\end{cases}
\quad
\text{\cite[Lemma 2.5(iv, v)]{FW10}}.
\end{align*}
(The sign can be determined explicitly by the results in \cite{FW10}.
It is irrelevant for our purpose
due to the exponent 2 in the decomposition \eqref{eq:decompKF},
so we omit it to simplify the presentation in calculations.)
Putting things together
with again the identity \eqref{eq:gauss_square}
and setting
\[ a_k(p) = \#\Xi_k(p) = \flr{\frac{k}{2p}+\frac{1}{2}}, \]
we obtain
\begin{align}
\nonumber
\eps(\PP^1_{(\inftyt)}, {j_*\sVt_k}_{(\inftyt)}, \omega_{(\inftyt)})
&= (g^{-k}p^{-2})^{k+1}(g^{-k+1}p^{-1})^{2a_k(p)} \\
\label{eq:eps_infinity}
&= \left(\frac{-1}{p}\right)^{(k+1)/2}
	(p^{-(k+1)/2})^{k+4+2a_k(p)}.
\end{align}

Putting \eqref{eq:eps_0} and \eqref{eq:eps_infinity} into \eqref{eq:eps_0_infinity},
we thus have
\[ \tilde{c} = \left(\frac{-1}{p}\right)^{(k+1)/2}
(-p^{(k+1)/2})^{k-2a_k(p)} \]
and hence the formula $c=(-1)^\delta\cdot\eqref{eq:det_odd}$
for $c = \tilde{c}/c^+$ when $k$ is odd.

\begin{rem}
One can also obtain the above results more directly
by investigating the structures of
$\sV_k = \qt\otimes\sV^+_k$ based on \cite{FW10}
instead of working on $\sVt_k$.
More details can be found in Appendix \ref{sect:Lim}.
\end{rem}

\section{Galois representations}\label{sect:twistedSFY}
In this section,
we construct, for each $k$,
the Galois representation of $\QQ$ associated with the twisted moments
$\{\tm_k(p)\}$
of the Kloosterman sums as a subquotient of the \'{e}tale cohomology
of an explicit affine variety.
We extract needed information from the work
\cite[\S\S 2 and 5.1]{FSY} in \S \ref{sect:motives},
which gives detailed properties of various realizations of the motives.
We prove the modularity Theorem \ref{thm:EC} in the last subsection
\S \ref{sect:modularity}.

\subsection{The motives}\label{sect:motives}
Let $\FF$ denote a field.
Over $\FF$,
let $\sK$ be the affine hypersurface
\[ \sK = \Big(\sum_{i=1}^k y_i + y_i^{-1} \Big) \subset \Gm^k \]
where $\{y_i\}$ is the Cartesian coordinate of the torus $\Gm^k$.
Let $\symgp_k$ be the $k$-th symmetric group,
acting on $\sK$ by permuting the variables $y_i$.
Let $\mu_2 = \{\pm 1\}$ be the finite group of order 2,
acting on $\sK$ by $y_i \mapsto \pm y_i$ for all $i$.
The two actions commute and yield an action of $\symgp_k\times\mu_2$
on $\sK$.
Consider the character
\[ \chi\colon \symgp_k\times\mu_2 \longrightarrow \{\pm 1\}, \]
defined to be the product of the sign character on $\symgp_k$
and the nontrivial one on $\mu_2$.
Let $\chi^+\colon \symgp_k\times\mu_2 \longrightarrow \{\pm 1\}$,
which is the sign on $\symgp_k$
and trivial on $\mu_2$.

Suppose $\FF=\FF_p$.
In \cite[Thm.3.12]{FSY} concerning the moment $\spm_k(p)$,
it is shown that
the \'etale cohomology $\coH_{\et,\rc}^1(\Gpb,\sV^+_k)$
is isomorphic to the Tate twist of the $\chi^+$-isotypic part
$\coH_{\et,\rc}^{k-1}(\sK_{\FFbar_p},\QQbar_\ell)_{\chi^+}(-1)$
of the \'etale cohomology of $\sK_{\FFbar_p}$
under the action of $\symgp_k\times\mu_2$,
and their weight $(k+1)$ parts also match.
Taking the $\chi$-isotypic parts instead in the following,
we obtain the corresponding statements for the twisted moments $\tm_k(p)$
by the same arguments.
Again the choices of $\chi^+$ and $\chi$ reflect
the decomposition in \eqref{eq:sheafV_relation}
under a double cover.

\begin{prop}\label{prop:V_K}
Take $\FF=\FF_p$ with $p\geq 3$.
We have
\begin{align*}
\coH^1_{\et,\rc}(\Gpb,\sV_k)
&\simeq
\coH^{k-1}_{\et,\rc}(\sK_{\FFbar_p},\QQbar_\ell)_\chi(-1), \\
\coH^1_{\et,\rmid}(\Gpb,\sV_k)
&\simeq
\big(\gr^W_{k-1}\coH^{k-1}_{\et,\rc}(\sK_{\FFbar_p},\QQbar_\ell)_\chi\big)(-1).
\end{align*}
Here $W$ denotes the weight filtration.
\end{prop}

Take $\FF=\QQ$.
The above proposition leads us to
define the motive over $\QQ$ associated with
the pure part of $\tm_k(p)$ as
\[ \Motive_k =
\big( \gr^W_{k-1}\coH_\rc^{k-1}(\sK)_\chi \big)(-1), \]
the Tate twist of the weight $k-1$ piece
of the $\chi$-eigenspace of the (various) middle cohomology
with compact support of $\sK$.
It is of rank
\[ \rk\Motive_k = \flr{\frac{k+1}{2}} - \delta_{2+4\ZZ}(k). \]
There is the natural pairing underlying the Poincar\'e pairing
\begin{equation}\label{eq:motive_pairing}
\Motive_k\otimes\Motive_k
\longrightarrow \coH^2(\PP^1)^{\otimes (k+1)} = \QQ(-k-1),
\end{equation}
which is $(-1)^{k+1}$-symmetric and perfect.

\subsubsection*{The \'etale realization}
For a prime number $\ell$, let
\[ V_{k,\ell}=\big(\gr^W_{k-1}\coH^{k-1}_{\et,\rc}(\sK_{\QQbar},\QQ_\ell)_\chi\big)(-1) \]
be the \'etale realization of $\Motive_k$,
furnished with the continuous action of the Galois group
\begin{gather*}
    \rho_{k,\ell}\colon \Gal(\QQbar/\QQ) \longrightarrow \GL(V_{k,\ell}).
\end{gather*}
The representation is compatible with the pairing induced by
\eqref{eq:motive_pairing}.

For $k$ odd,
we let
\begin{equation}\label{eq:Theta}
\Theta_k(p)
= \{ \text{$j$ odd positive integer} \mid jp\leq k \}.
\end{equation}
Let $v_p\colon \ZZ\setminus\{0\} \to \ZZ$
be the $p$-adic valuation.
We also set
$\Theta_k(p) = \Theta_k^+(p) \sqcup \Theta_k^-(p)$
where
\begin{equation}\label{eq:Thetapm}
	 \Theta_k^+(p) = \{j\in\Theta_k(p) \mid \text{$v_p(j)$ odd} \},
\quad
\Theta_k^-(p) = \{j\in\Theta_k(p) \mid \text{$v_p(j)$ even} \}. 
\end{equation}
We mention that there is a bijection
$\Xi_k(p) \xrightarrow{\sim} \Theta_k(p)$
given by
$i \mapsto (k-2i)/p$
where the set $\Xi_k(p)$ is defined in \eqref{eq:Xi}.

\begin{prop}\label{oddrep}
Let $k$ be a positive odd integer.
For any prime $\ell$,
the $\ell$-adic \'etale realization $V_{k,\ell}$ of $\Motive_k$
is a quadratic space and
yields a continuous representation
\begin{gather*}
    \rho_{k,\ell}\colon \Gal(\QQbar/\QQ) \longrightarrow \mathrm{GO}(V_{k,\ell})
\end{gather*}
satisfying the following properties.
\begin{enumerate}
\item
Let $p$ be an odd prime, $p\neq\ell$.
Fix a place of $\QQbar$ above $p$,
and $\Gal(\QQbar_p/\QQ_p) \subset \Gal(\QQbar/\QQ)$
the corresponding decomposition subgroup.
As a representation of $\Gal(\QQbar_p/\QQ_p)$,
$V_{k,\ell}$ decomposes into an orthogonal sum
$M\oplus E$ fulfilling the following conditions.
\begin{itemize}
\item
One has
$M\otimes\QQbar_\ell \cong \coH_{\et,\rc}^1(\Gpb,\sV_k)$.
\item
$E$ is a direct sum
$\bigoplus_{a\in\Theta_k(p)}\eps_a\otimes\chi_\cyc^{-(k+1)/2}$
where $\eps_a$ is the primitive character associated with the extension
$\QQ_p(\sqrt{(-1)^{(1+ap)/2}2ap})$ of $\QQ_p$
(hence $\eps_a$ is unramified precisely when $a\in\Theta_k^+(p)$).
\end{itemize}
\item
We have
\begin{equation}\label{eq:det_V}
\det \rho_{k,\ell} = \chi_D\cdot 
\chi_\cyc^{-(k+1)^2/4},
\quad
D = \begin{cases}
(-1)^{(k+3)/4}\cdot 4\cdot k!!,
& k\equiv 1\bmod{4}, \\
(-1)^{(k+1)/4}\cdot k!!,
& k\equiv -1\bmod{4},
\end{cases}
\end{equation}
where $\chi_D$ is the quadratic character associated with
the extension $\QQ(\sqrt{D})$ over $\QQ$.
\end{enumerate}
\end{prop}

\begin{prop}\label{evenrep}
Let $k$ be a positive even integer.
For any prime $\ell$,
the $\ell$-adic \'etale realization $V_{k,\ell}$ of $\Motive_k$
is a symplectic space
and yields a continuous representation
\begin{gather*}
    \rho_{k,\ell}\colon \Gal(\QQbar/\QQ) \longrightarrow \mathrm{GSp}(V_{k,\ell}).
\end{gather*}
Let $p$ be an odd prime, distinct from $\ell$.
Fix a place of $\QQbar$ above $p$
and let $I_p$ be the corresponding inertia subgroup
of $\Gal(\QQbar_p/\QQ_p) \subset \Gal(\QQbar/\QQ)$.
Then for all $\sigma\in I_p$,
$(\rho_{k,\ell}(\sigma )-1)^2 = 0$ on $V_{k,\ell}$.
Moreover,
there is an isotropic subspace $U$ of $V_{k,\ell}$,
$\dim U = \flr{k/2p}$,
such that the induced map
$\rho_{k,\ell}(\sigma) -1\colon V_{k,\ell} \to V_{k,\ell}/U$ is zero
for $\sigma\in I_p$ and
$V_{k,\ell}^{I_p} = U^\perp$.
As representations of $\Gal(\FFbar_p/\FF_p)$,
we have
\[ V_{k,\ell}^{I_p}\otimes\QQbar_\ell \cong
\coH_{\et,\rc}^1(\Gpb,\sV_k)/E,
\quad
E \cong \QQbar_\ell(-k/2)^{\delta_{2+4\ZZ}(k)}. \]
\end{prop}

The above two propositions are the twisted analogues
of the corresponding Thm.5.8 and Thm.5.17 for $\spm_k(p)$ in \cite{FSY}.
The proofs rely on carefully applying the Picard-Lefschetz formula
on (an explicit compactification of) $\sK$, regarded as defined over $\ZZ_p$.
This $\sK$ has additional ordinary quadratic singularities in its reduction,
and yields vanishing cycle classes in
$\coH_{\et,\rc}^{k-1}(\sK_{\QQbar_p},\QQ_\ell)$,
which are complementary to $\coH_{\et,\rc}^{k-1}(\sK_{\FFbar_p},\QQ_\ell)$
in a precise sense.
Taking the $\chi$-isotypic parts instead of $\chi^+$,
the same arguments loc.\,cit.\,verbatim give the proofs for the $\tm_k(p)$ case,
except the formula \eqref{eq:det_V} for the determinant.
The latter is an immediate consequence of Proposition \ref{prop:V_K},
the determinant formula \eqref{eq:det_odd},
applying to unramified primes $p>k$,
and the Chebotarev density theorem. 

\begin{prop}\label{prop:p-adicHodge}
The representation $V_{k,\ell}$ is de Rham at $\ell$;
when $k$ is odd, it is crystalline if $\ell$ is odd and $\ell >k$;
when $k$ is even, it is crystalline if $\ell >k/2$.
\end{prop}

\begin{proof}
Note that $V_{2,\ell} = 0$,
so the assertion is empty when $k=2$.
Otherwise, this is the corresponding statement
of \cite[Prop.5.23]{FSY} for $\spm_k(p)$,
which is proved using the geometry of $\sK$
and the $p$-adic Hodge comparison theorems.
One argues verbatim as loc.\,cit.,
taking the $\chi$-isotypic parts instead of $\chi^+$.
\end{proof}

\subsubsection*{The de Rham realization}
The de Rham realization of the motive $\Motive_k$
is the middle part of the de Rham cohomology
\[ V_{k,\dR} = \big(\gr^W_{k-1}\coH^{k-1}_{\dR,\rc}(\sK/\QQ)_\chi\big)(-1). \]
It is a $\QQ$-vector space equipped with a Hodge filtration.
By \cite[Prop.4.21]{FSY} stating the Hodge numbers
for the motives underlying $\spmt_k(p)$,
and subtracting \cite[Thm.1.8]{FSY}
for the Hodge numbers of $\spm_k(p)$,
we obtain the Hodge numbers of $V_{k,\dR}$ below.
We remark that,
unlike the \'etale realizations $V_{k,\ell}$
where the detailed descriptions as Galois representations
are revealed through
the geometry of $\sK$,
the Hodge structures $V_{k,\dR}$ are investigated
mainly through an irregular connection on $\Gm$ over $\CC$,
which is the complex counterpart of $\sV_k$,
via irregular Hodge theory
(also see \cite{Chuang} for the period realization of the connection).
The weight part of the statement can also be read off
from Theorem \ref{thm:GR}.

\begin{prop}\label{prop:HodgeNumbers}
Let
$V'_{k,\dR} =
\coH^{k-1}_{\dR,\rc}(\sK/\QQ)_\chi(-1)$;
it has (Hodge) weights $\leq k+1$
and contains $V_{k,\dR}$ as its top weighted quotient.
\begin{enumerate}
\item
When $k\not\equiv 2\bmod{4}$,
we have $V'_{k,\dR} \xrightarrow{\sim} V_{k,\dR}$
and they are of dimension $\flr{(k+1)/2}$.
The Hodge numbers
$h^{p,k+1-p} = \dim\gr_F^pV'_{k,\dR} = 1$
if $p = 2i-1$ or $k+2-2i$ for $1\leq i\leq \flr{(k+3)/4}$,
and vanish otherwise.
\item
If $k\equiv 2\bmod{4}$,
then $\dim V'_{k,\dR} = k/2$
and $\dim V_{k,\dR} = (k-2)/2$.
The Hodge numbers
$h^{p,q} = \dim\gr^W_{p+q}\gr_F^pV'_{k,\dR}$ vanish
except in the following cases, in which $h^{p,q} = 1$,
\[\begin{cases}
\text{$p$ or $q = 2i-1$ for $1\leq i\leq (k-2)/4$, and $p+q=k+1$}, \\
p=q=k/2.
\end{cases}\]
\end{enumerate}
\end{prop}

\begin{cor}\label{cor:pot_auto}
The representation $\rho_{k,\ell}$
is potentially automorphic.
\end{cor}

\begin{proof}
As in \cite[\S 5.3.3]{FSY},
we apply the main automorphy result of \cite{PT}
to the weakly compatible system $\{\rho_{k,\ell}\}_\ell$ for $k$ fixed.
By Proposition \ref{prop:p-adicHodge},
$\rho_{k,\ell}$ is de Rham
above $\ell$,
and is crystalline when $\ell$ is large enough.
We have the three properties.
\begin{itemize}
\item[(i)] (Purity)
For $p$ odd and $p>k$ if $k$ is odd (resp., $p>k/2$ if $k$ is even),
$\rho_{k,\ell}$ is unramified at $p$,
and the roots of $\rho_{k,\ell}(\Frob_p)$ of the geometric Frobenius 
are Weil $p^{(k+1)}$-numbers.
\item[(ii)] (Regularity)
The nonvanishing Hodge numbers $h^{i,j}$
of (the underlying motive of) $\rho_{k,\ell}$
are equal to one.
\item[(iii)] (Odd essential self-duality)
$\rho_{k,\ell}$ is $(-1)^{k+1}$-symmetric
with similitude character $\chi_\cyc^{-k-1}$;
the determinant $\{\det\rho_{k,\ell}\}_\ell$
also forms a weakly compatible system.
\end{itemize}
Therefore, by \cite[Thm.A]{PT},
$\rho_{k,\ell}$ is potentially automorphic.
\end{proof}

Recall the polynomial $Z_k(p;T)$,
the set $\Theta_k(p)$ and the numbers $n_k^\pm(p)$
defined in \eqref{eq:localfactor_Z},
\eqref{eq:Theta} and \eqref{eq:n_k}, respectively.
Let $\BB_\crys$ and $\BB_\st$ be
Fontaine's $p$-adic crystalline and semistable period rings over $\QQ_p$.
They are filtered rings with an absolute Frobenius
(semilinear) automorphism $\varphi$.
We have the following strengthening of Proposition \ref{prop:p-adicHodge}.

\begin{cor}
Let $p$ be an odd prime.
Choose a place of $\QQbar$ above $p$,
and consider $V_{k,p}$ as a representation
of the corresponding decomposition subgroup $\Gal(\QQbar_p/\QQ_p)$.
\begin{enumerate}
\item
Suppose $k$ is odd.
Then $V_{k,p}$ is crystalline over $L = \QQ_p(\sqrt{-p})$
and
\begin{equation}\label{eq:det_phi_odd}
\det\big(1-\varphi T\mid (V_{k,p}\otimes \BB_\crys)^{\Gal(\QQbar_p/L)}\big)
= Z_k(p;T)\prod_{a\in\Theta_k(p)}
\Big( 1- \Big(\frac{(-1)^{(1+ap)/2}2a'}{p}\Big) p^{\frac{k+1}{2}}T \Big)
\end{equation}
where
$a' = ap^{-v_p(a)}$ is the prime-to-$p$ part of $a$.
If $p>k$, $V_{k,p}$ is crystalline over $\QQ_p$
and one can replace $L$ by $\QQ_p$ in the equation.
\item
Suppose $k$ is even.
Then $V_{k,p}$ is semistable over $\QQ_p$
and
\begin{align}\label{eq:det_phi_even}
\det\big(1-\varphi T\mid
	(V_{k,p}\otimes \BB_\st)^{\Gal(\QQbar_p/\QQ_p)}\big)
&= Z_k(p;T)\big(1-p^{\frac{k}{2}}T\big)^{-\delta_{2+4\ZZ}(k)} \\
&\nonumber \hspace{20pt}\cdot
\big(1-p^{\frac{k+2}{2}}T\big)^{n_k^+(p)}
\big(1-(-1)^{\frac{p-1}{2}}p^{\frac{k+2}{2}}T\big)^{n_k^-(p)}.
\end{align}
If $p>k/2$, $V_{k,p}$ is crystalline over $\QQ_p$
and one can replace $\BB_\st$ by $\BB_\crys$ in the equation.
\end{enumerate}

In particular,
the $p$-adic Newton polygons of
the polynomials \eqref{eq:det_phi_odd} and \eqref{eq:det_phi_even}
are above the Hodge polygon of $V_{k,\dR}$,
and the polygons have the same endpoints.
\end{cor}

\begin{proof}
(Cf.\,\cite[Rem.5.41]{FSY}.)
A consequence of the potential automorphy Corollary \ref{cor:pot_auto}
is that the system $\{\rho_{k,\ell}\}_\ell$ is indeed strictly compatible
(\cite[Cor.2.2(ii)]{PT}).
Thus the Weil-Deligne representation $(r_\ell,\rN)$
of $\QQ_p$
associated with $V_{k,\ell}$ for $\ell\neq p$
(cf.\,\cite[\S 5.3.1]{FSY})
equals that of $V_{k,p}$
constructed via the $p$-adic period rings (Fontaine's functors).

Suppose $k$ is odd.
By Proposition \ref{oddrep}(i),
$V_{k,\ell}$ is unramified over $L$
and $\det(1-\Frob_p T\mid V_{k,\ell})$
equals the right-hand side of \eqref{eq:det_phi_odd}.
Here we regard $\Frob_p$
as in the quotient of $\Gal(\QQbar_p/L)$.

Assume $k$ is even.
By Proposition \ref{evenrep},
the associated Weil representation $r_\ell$ is unramified.
Let $\{e_i\}$ be a basis of the subspace $U \subset V_{k,\ell}^{I_p}$
in Proposition \ref{evenrep}.
Then there are elements $\{e_i'\}$, inducing a basis of $V_{k,\ell}/U^\perp$,
such that $\rN e_i' = e_i$.
Since $\rN\circ \Frob_p = p\Frob_p\circ \rN$
and
\[ \det(1-\Frob_pT\mid U)
= (1-p^{k/2}T)^{n_k^+(p)} (1-(-1)^{(p-1)/2}p^{k/2}T)^{n_k^-(p)} \]
by Theorem \ref{thm:zeta_at_p}(ii)(a),
the equality \eqref{eq:det_phi_even} follows.

On the other hand,
the associated filtered $\varphi$-module of $V_{k,p}$
are (weakly) admissible.
Therefore the attached Newton polygon lies above the Hodge polygon,
and the two have the same endpoints.

We remark that one can deduce weaker statements
using rigid cohomology and $p$-adic Hodge theory more directly,
cf.\,\cite[Prop.5.23, Cor.5.27]{FSY}.
\end{proof}

\subsection{Modularity}\label{sect:modularity}
Here we prove Theorem \ref{thm:EC}.
In the cases $k=3,4,6$,
the associated Galois representations
$\rho_{k,\ell}$ of $\Motive_k$ are of rank 2.
The modularity is obtained similarly to that in \cite[\S 4.8]{Serre-GQ}.
For the rank 3 representation $\rho_{5,\ell}$,
we argue as in \cite[\S 4.7]{Yun}.

In all cases,
the exponents $c$ of 2 in the level $N$
of the associated cuspidal modular forms
are at most $8$
by the last paragraph of \cite[p.216]{Serre-GQ}.

Let $\chi_{\cyc,\ell}$ denotes the $\ell$-adic cyclotomic character
of $\Gal(\QQbar/\QQ)$.

\subsubsection*{The case $k=3$}
Due to the existence of the symmetric perfect pairing
\eqref{eq:motive_pairing},
the orthogonal representation
$\rho_{3,\ell}\otimes\chi_{\cyc,\ell}$ is modular (with CM)
of Serre weight 3
by \cite[Thm.1.3]{Livne}.
Let $f$ be the corresponding cusp form.
By a theorem of Fontaine
(\cite[Thm.2.6]{Edixhoven}),
the reduction
$\overline{\rho}_{3,\ell}\colon \Gal(\QQbar/\QQ) \to \GL(\FF_\ell)$
mod $\ell$ of $\rho_{3,\ell}$ is absolutely irreducible
for those $\ell$ such that the Fourier coefficient $a_f(\ell) = 0$
(there are infinitely many since $f$ has CM).
Now we apply the strong form of Serre's modularity conjecture
to pin down $f$.
By the recipe in \cite{Serre-GQ},
we look for a normalized Hecke eigenform
with coefficients in $\QQ$
among those of weight $3$
and level $2^c\cdot 3$ for some $0\leq c\leq 8$.
In \texttt{{\LMFDB}},
it results in 8 matches.
The conditions
\[ a_f(7) = -\tm_3(7)/7 = -2,
\quad
a_f(13) = -\tm_3(13)/13 = -22 \]
($a_f(5)=a_f(11)=0$)
single out one such form $f_{\text{48.3.e.a}}$
(which has CM by $\QQ(\sqrt{-3})$).

\begin{rem}\label{rem:Evans3}
Let
$g = f_{\text{12.3.c.a}} = f_{\text{48.3.e.a}}\otimes\big(\frac{-4}{\bullet}\big)$.
It is twist minimal and satisfies
\[ a_g(p) = (-1)^{(p-1)/2}\tm_3(p)/p. \]
The equality is the statement in \cite[p.523]{Evans-hyper}.
\end{rem}

\subsubsection*{The case $k=4$}
Applying \cite[Th.6]{Serre-GQ} directly
to the symplectic representation $\rho_{4,\ell}\otimes\chi_{\cyc,\ell}$,
we look for a normalized Hecke eigenform
among those of weight $4$, level $2^c$ for some $0\leq c\leq 8$,
and with coefficients in $\QQ$.
In \texttt{{\LMFDB}},
it results in 22 matches.
The conditions
\[ a_f(3) = -\tm_4(3)/3 = -4,
\quad
a_f(5) = -\tm_4(5)/5 = -2 \]
single out one such form $f_{\text{8.4.a.a}}$
We remark that this form is twist minimal.

\subsubsection*{The case $k=6$}
Applying directly \cite[Th.6]{Serre-GQ}
to the symplectic representation $\rho_{6,\ell}\otimes\chi_{\cyc,\ell}$,
we look for a normalized Hecke eigenform
among those of weight $6$, level $2^c\cdot 3, 0\leq c\leq 8$,
with coefficients in $\QQ$.
In \texttt{{\LMFDB}},
it results in 48 matches.
The conditions
\[ a_f(5) = -\tm_6(5)/5 - 5^2 = -34,
\quad
a_f(7) = -\tm_6(7)/7 - 7^2 = -240 \]
single out one such form $f_{\text{24.6.a.a}}$.
We remark that this form is twist minimal.

\subsubsection*{The case $k=5$} 
Consider the Galois representations
$\rho_{5,\ell}$
constructed in Theorem \ref{oddrep}.
Let
\begin{equation}\label{eq:repre_r}
r_{\ell}=\rho_{5,\ell}\otimes \chi_{\cyc,\ell}^3\otimes\Bigl(\frac{15}{\bullet}\Bigr)\colon
\Gal (\QQbar/\QQ) \longrightarrow \SO(V_{5,\ell}).
\end{equation}
It satisfies the following properties:
\begin{itemize}
\item
The representations $r_\ell$ is unramified at $p$
for every prime $p\neq 2, 3, 5, \ell$.
In this case,
the trace of the geometric Frobenius is given by
\begin{equation}\label{eq:trace_r_ell}
\tr (r_\ell(\Frob_p))=\Bigl(\frac{15}{p}\Bigr)\frac{-\tm_5(p)}{p^3}.
\end{equation}
\item
It is tamely ramified at $p=3,5$.
\item
It is de Rham at $p=\ell$,
and has Hodge-Tate weight $(2,0,-2)$.
It is indeed crystalline if $\ell >5$.
\item
The image $r_\ell(\iota)$ of the complex conjugation $\iota$ is
conjugate to $\diag(1,-1,-1)$.
\end{itemize}

Fix the choices of the identifications
\[ \SO(V_{5, \ell}\otimes_{\QQ_\ell}\QQbar_\ell)
\simeq \SO(3)(\QQbar_\ell)\simeq \PGL_2(\QQbar_\ell). \]
Let
\[ r_\ell'\colon \Gal (\QQbar/\QQ) \longrightarrow \PGL_2(\QQbar_\ell) \]
be the \textit{same} representation as $r_\ell$
under the identifications.
In particular,
\begin{equation}\label{eq:r'odd}
r'_\ell(\iota) \sim \begin{pmatrix}
1 \\
& -1 \end{pmatrix}.
\end{equation}

\begin{lem}\label{ControlFrob}
Suppose $\ell >5$.
For $p=3, 5$,
the representation $r'_\ell$ is tamely ramified at $p$,
and it sends a generator of the tame quotient $I_p^t$
of the inertia subgroup
$I_p \subset \Gal(\QQbar_p/\QQ_p)$
to a matrix conjugate to
$\diag(1,-1)$
(up to scaling).
Moreover,
the image $r'_\ell(\Gal(\QQbar_p/\QQ_p))\subset \PGL_2(\QQbar_\ell)$
is conjugate to a subgroup of the diagonal matrices
(up to scaling).
\end{lem}

\begin{proof}
Since $r_\ell$ is tamely ramified at odd primes,
the same holds for $r_\ell'$.
For $p=3, 5$,
choose a topological generator $\sigma$ of $I_{p}^t$.
Since $\dim_{\QQ_\ell}V_{5,\ell}^{I_p}=2$,
$\rho_\ell(\sigma)$ is similar to $\diag(1,1,-1)$.
As a result, $r_\ell(\sigma)$ is similar to $\diag(-1,-1,1)$
in $\SO(3)({\QQbar_\ell})$,
and hence $r'_\ell(\sigma)$ is similar to
$\diag(1,-1)$ in $\PGL_2({\QQbar_\ell})$.
In the following,
we choose a frame such that
$r'_\ell(\sigma) = \diag(1,-1)$
(up to scaling).

Let $\phi_p \in \Gal(\QQbar_p/\QQ_p)$
be a lifting of the geometric Frobenius.
We have $\phi_p^{-1}\sigma\phi_p=\sigma^p$ in $I_{p}^t$,
and hence $r'_\ell(\phi_p)$ is equal to 
$\bigl( \begin{smallmatrix}
a&0\\ 0&d
\end{smallmatrix} \bigr)$
or
$\bigl( \begin{smallmatrix}
0&b\\ c&0
\end{smallmatrix} \bigr)$.
In these two cases,
$r_\ell(\phi_p)$ are then similar to
\[ \begin{pmatrix}
ad^{-1} \\
& a^{-1}d \\
&& 1 \end{pmatrix},
\quad
\begin{pmatrix}
& bc^{-1} \\
b^{-1}c \\
&& -1 \end{pmatrix}, \]
respectively,
in which the actions on the two-dimensional subspace
$(\QQbar_\ell^3)^{\sigma +1}$ occor in the upper-left parts.
Observe that
$\det(1-F_pT\mid \coH^1_{\et,\rc}(\Gpb,\sV_5)^{I_p}) = Z_5(p;T)$ are 
$$
  1+6T+729T^2 \quad \mbox{and} \quad  1+150T+15625T^2
$$ 
for $p=3$ and $5$, respectively.
This implies in particular that
$\tr(r_\ell(\phi_p)\mid (\QQbar_\ell^3)^{\sigma +1}) \neq 0$.
Hence, $r'_\ell(\phi_p) = \diag(a,d)$,
and the proof is completed.
\end{proof}

\begin{prop}[Cf.\,{\cite[Lemma 4.7.6]{Yun}}]
\label{lift}
For every prime $\ell > 5$, there exists a lifting
\[ \widetilde{r}_\ell \colon \Gal(\QQbar/\QQ) \longrightarrow \GL_2(\QQbar_\ell) \]
of $r'_\ell\colon \Gal(\QQbar/\QQ) \to \PGL_2(\QQbar_\ell)$
satisfying the following properties:
\begin{itemize}
    \item For prime $p\neq 2, 3, 5, \ell$, the representation $\widetilde{r}_\ell$ is unramified at $p$.
    \item For prime $p=3, 5$, the representation $\widetilde{r}_\ell$ is tamely ramified at $p$,
and it sends a generator of $I_p^t$
to a matrix conjugate to $\diag(1,-1)$.
\item
$\widetilde{r}_\ell$ is crystalline at $\ell$
and its Hodge-Tate weight equals $(2,0)$.
\item
$\det\widetilde{r}_\ell = \big(\frac{-15}{\bullet}\big)\otimes\chi_{\cyc,\ell}^{-2}$.
\end{itemize}
We then have
$r_\ell = \Sym^2\widetilde{r}_\ell \otimes \det^{-1}\widetilde{r}_\ell$
over $\QQbar_\ell$.
\end{prop}

\begin{proof}
By \cite[Prop.5.5]{Pat15},
there exists a lifting
$\Phi\colon \Gal(\QQbar/\QQ) \to \GL_2(\QQbar_\ell)$ of $r'_\ell$,
which is unramified almost everywhere
and is de Rham at $\ell$.

For $p\neq 2, 3, 5, \ell$,
we have $\Phi|_{I_p} = c_p\mathrm{Id}$
for some character $c_p$ of $I_p$
since $r'_\ell|_{I_p}$ is trivial.
Note that $c_p$ are trivial for almost all $p$ since $\Phi$ is unramified almost everywhere.

For $p=3,5$,
the representation $\Phi|_{\Gal(\QQbar_p/\QQ_p)}$
is isomorphic to $d_p\oplus d_p^{\prime}$
by Lemma \ref{ControlFrob}.
The restrictions of $d_p$ and $d_p^{\prime}$ at the wild inertia group are the same
since $r_\ell'$ is tamely ramified.
We denote $c_p$ the restriction of $d_p$ to the inertia group.

Since $r_\ell'$ is crystalline
and the lifting $\Phi$ is de Rham at $\ell$,
by \cite[Cor.6.7]{Conrad},
there exists a crystalline lifting $\phi$ of
$r'_\ell|_{\Gal(\QQbar_\ell/\QQ_\ell)}$.
We may assume that $\phi$ has Hodge-Tate weight $(2,0)$
after a suitable Tate twist.
Then $\Phi|_{\Gal(\QQbar_\ell/\QQ_\ell)} = \phi\cdot d_\ell$
for a continuous character
$d_\ell\colon \Gal(\QQbar_\ell/\QQ_\ell) \to \QQbar_\ell^\times$.
Take $c_\ell = d_\ell|_{I_\ell}$.

At $p=2$,
identify $I_2$ with $\{\pm 1\}\times\ZZ_2$
where $\{\pm 1\} = \Gal(\QQ_2(\sqrt{-1})/\QQ_2)$.
The restriction $\Phi|_{I_2}$ is quasi-unipotent,
and hence $\det\Phi$ factors through a finite quotient
$\{\pm 1\}\times \ZZ/2^a$.
Pick an $\alpha \in \QQbar_\ell^\times$
such that $\det\Phi|_{I_2}(1,1) = \alpha^2$.
Let $c_2: I_2 \to \QQbar_\ell^\times$
given by $c_2|_{\{\pm 1\}} = 1$
and $c_2(1,1) = \alpha$
so that $c_2^{-2}\det\Phi|_{I_2}$
factors through $\{\pm 1\}$.

Let
$c\colon \Gal(\QQbar/\QQ)^{\mathrm{ab}}= \prod_p \ZZ_p^{\times}\to\QQbar_\ell^{\times}$ be the product of all $c_p$
via the identification $I_p = \ZZ_p^\times$
(permuting $p^m$-th roots of unity for $m\geq 1$).
This $c$ is well-defined, as $c_p$ is trivial for almost every $p$.
Then $\widetilde{r}_\ell=\Phi\otimes c^{-1}$ is a lifting
satisfying the first three listed conditions
with the twist of the determinant
$(\det\widetilde{r}_\ell)\otimes \chi_{\cyc,\ell}^2$ equal to
either $\big(\frac{-15}{\bullet}\big)$ or $\big(\frac{60}{\bullet}\big)$
since it has conductor 15 or 60.
Since $\widetilde{r}_\ell$ is odd by \eqref{eq:r'odd},
one has
$(\det\widetilde{r}_\ell)\otimes \chi_{\cyc,\ell}^2= \big(\frac{-15}{\bullet}\big)$.
\end{proof}

Consider the reduction mod $\ell$
\[ \overline{\widetilde{r}}_\ell\colon
\Gal(\QQbar/\QQ) \longrightarrow \GL_2(\FFbar_\ell) \]
of $\widetilde{r}_\ell$.
It is well-defined up to semisimplification.

\begin{lem}\label{lem:abs_irred}
The reduction $\overline{\widetilde{r}}_\ell$ is irreducible
for $\ell\gg 1$.
\end{lem}
\begin{proof}
(Cf.\,\cite[(4.8.9)]{Serre-GQ},
\cite[Thm.1.2.1]{Kisin},
\cite[Lemma 4.7.2]{Yun}
and their proofs.)
Suppose, on the contrary, that there are infinitely many primes $\ell$
such that $\overline{\widetilde{r}}_\ell$
is not irreducible.
So, for these $\ell$,
there are characters
$\psi_1,\psi_2\colon \Gal(\QQbar/\QQ) \to \FFbar_\ell^\times$
such that
$\overline{\widetilde{r}}_\ell \simeq \psi_1\oplus\psi_2$
up to semisimplification.
Then,
for the reduction $\overline{r}_\ell$ of $r_\ell$ in \eqref{eq:repre_r},
we have
$\overline{r}_\ell \simeq \psi_1\psi_2^{-1}\oplus 1\oplus \psi_1^{-1}\psi_2$.
Let $N$ be the conductor of $r_\ell$.
Since $r_\ell$ is de Rham at $\ell$,
there is a Dirichlet character
$\chi_0\colon \Gal(\QQbar/\QQ) \to \QQbar^\times\subset\QQbar_\ell^\times$
of conductor dividing $N$
such that $\psi_1\psi_2^{-1} = \overline{\chi}$
where $\overline{\chi}$ is the reduction of $\chi = \chi_{\cyc,\ell}^{-2}\chi_0$
(after exchanging $\psi_1,\psi_2$ if needed).
In particular,
by the evaluation of
$\rho_{5,\ell} = r_\ell\chi_{\cyc,\ell}^{-3}\big(\frac{15}{\bullet}\big)$
at $\Frob_p$
for $p>5, p\neq\ell$,
one has
\begin{equation}\label{eq:m5_cong}
-\Bigl(\frac{15}{p}\Bigr)\tm_5(p)
\equiv p^3\big(p^2\chi_0(p) + 1 + p^{-2}\chi_0^{-1}(p)\big) \bmod\ell.
\end{equation}
Since the congruence holds for arbitrarily large $\ell$
and the involved characters $\chi_0$ have bounded conductors,
there exists a fixed such $\chi_0$
making \eqref{eq:m5_cong}
an equality in the ring of integers of a fixed cyclotomic field
(adding values of $\chi_0$ to $\QQ$).

On the other hand,
the integer $\tm_5(p)$
is in the range $[-3p^3,3p^3]$
since $\rho_{5,\ell}$ is of rank 3 and weight 6.
However,
one has
\[ |p^5\chi_0(p) + p^3 + p\chi_0^{-1}(p)| \geq p^5-p^3-p, \]
and the right-hand side is greater than $3p^3$ for $p\geq 5$,
yielding a contradiction.
\end{proof}

\begin{lem}\label{lemma:tm5}
For $p\equiv 5 \bmod 6$, we have
	$$
\frac{-\tm_5(p)}p  +\Bigl(\frac{15}{p}\Bigr) p^2 
 \quad	\in  5\mathbb Z.
	$$
	Furthermore, up to $p=2357$ (the $350$th prime), the values  $\tm_5(p)$ are of the form
	\begin{align*}
		\Bigl(\frac{-5}{p}\Bigr) \frac{-\tm_5(p)}p+\Bigl(\frac{-3}{p}\Bigr) p^2 =
		\begin{cases}
			m_p^2, & \mbox{ if } p \equiv 1 \bmod 6,\\
			-5n_p^2, & \mbox{ if } p \equiv 5 \bmod 6,
		\end{cases}
	\end{align*}
	where $m_p$ and $n_p \in \ZZ$.  
	
\end{lem}
\begin{proof}
Immediately from the definition, one has
$$
  \Kl_2^{\Sym^5}(a;p) 
  =  \Kl_2^{\Sym^1}(p;a)^5 -4p  \Kl_2^{\Sym^3}(a;p) -5\Kl_2^{\Sym^1}(p;a)p^2.
  $$
By the direct evaluation \eqref{eq:small_moments} of $\tm_1(p)$,
the modularity of $\tm_3(p)$ in \eqref{eq:Evans3}
and that $f=f_{\text{48.3.e.a}}$ has CM by $\QQ(\sqrt{-3})$, we can write 
  \begin{align*}
  	\tm_5(p) =& -\sum_{a\in\FF_p^\times}\Big(\frac ap\Big)\Kl_2(p;a)^5+4p^2 \cdot a_f(p) +5\Big(\frac{-1}p\Big)p^3\\
  	          =& -\sum_{a\in\FF_p^\times}\Big(\frac ap\Big)\Kl_2(p;a)^5 +5\Big(\frac{-1}p\Big)p^3,
  \end{align*}
  for  $p\equiv 5 \bmod 6$.   	Observe that 
	   $$
	  \frac{-\tm_5(p)}p \equiv \frac 1p  \sum_{a\in\FF_p^\times}\Big(\frac ap\Big)\Kl_2(p;a)^5 \equiv \frac 1p \sum_{a\in\FF_p^\times}\Big(\frac ap\Big)\Kl_2(p;a)  = \Bigl(\frac{-1}{p}\Bigr) \equiv   \Bigl(\frac{-5}{p}\Bigr) p^2\bmod 5,
	  $$
	  since  $p^2 \equiv  \bigl(\frac{5}{p}\bigr) \bmod 5$. 

The second assertion can be derived from common software.
\end{proof}

We are now ready to prove the modularity of $\tm_5$.
For a given prime $\ell >5$,
let
$r_\ell=\rho_{5,\ell}\otimes \chi_{\cyc,\ell}^3\otimes\bigl(\frac{15}{\bullet}\bigr)$.
Then
$r_\ell = \Sym^2\widetilde{r}_\ell \otimes \det^{-1}\widetilde{r}_\ell$
for a degree-$2$ representation $\widetilde{r}_\ell$
as in Proposition \ref{lift}.

Let $\bar{\widetilde{r}}_\ell$ be the mod $\ell$ representation of $\widetilde{r}_\ell$, well-defined up to semisimplification.
We may assume that $\bar{\widetilde{r}}_\ell$ is absolutely irreducible
by choosing $\ell\gg 0$, by Lemma \ref{lem:abs_irred}.  
The Artin conductor of $\widetilde{r}_\ell$ is $2^c\cdot3\cdot5$,
for some $c \in \ZZ_{\geq 0}$,
by Proposition \ref{lift}. By Proposition \ref{prop:HodgeNumbers},
we deduce that the Serre weight of  $\widetilde{r}_\ell$ is $3$.
The assertion of Serre's modularity conjecture
tells us that there exists a newform $f$ of level $2^c\cdot15$
for some $0\leq c\leq 8$,
weight $3$, character $\epsilon_f$ with $q$-expansion $\sum a_f(n)q^n$ and associated Galois representation $\rho_f$ such that the reduction mod $\ell$ representation of $\rho_f$ is isomorphic to $\bar{\widetilde{r}}_\ell$.

By Proposition \ref{lift}, the character $\epsilon_f$ is induced from $\(\frac{-15}{\bullet}\)$.
Therefore, the target Hecke eighenform $f$ satisfies that
\begin{gather*}
    a_f(p)^2-{\Big(\frac{-15}{p}\Big)}p^2
    = {\Big(\frac{-15}{p}\Big)} p^2\cdot \tr (r_\ell(\Frob_p))
    = {\Big(\frac{-1}{p}\Big)}  \frac{-\tm_5(p)}{p},
\end{gather*} 
for prime $p \geq 7$. 
Further by Lemma \ref{lemma:tm5} for the case of $p\equiv 5\bmod 6$, we must have
$$
    \(\frac{a_f(p)}{\sqrt{-5} n_p}\)^2={\Big(\frac{-15}{p}\Big)\Big(\frac{-3}p\Big)
    =\Big(\frac{5}{p}\Big)}, \quad \mbox{for some } n_p\in \ZZ.
$$
Also, for $p\equiv 1\bmod{6}$ with  $p\leq 2357$, we have 
	\[ 
	\(\frac{a_f(p)}{m_p}\)^2=\Big(\frac{-15}{p}\Big)\Big(\frac{-3}p\Big)
    =\Big(\frac{5}{p}\Big), \quad \mbox{for some } m_p\in \ZZ.
    \]
	So
	$a_f(p) = \pm m_p, \pm m_pi, \pm\sqrt{5}n_p, \pm\sqrt{-5}n_p$. Since  the Sturm bound for the  space of modular form of level $2^8\cdot 15$ with a given character is $2304$,  we conclude that the coefficient field $K_f:=\QQ(a_f(n) : n \geq  1)$ is $\QQ(\sqrt{5},i)$. 
	
	Taking $f$ to be any of the normalized newform in the space $S_3\big(60,  \big(\frac{-15}{\bullet}\big)\big)$ with $K_f = \QQ(\sqrt{5},i)$, one can check the validity of the identity numerically
	$$
	a_f(p)^2=\Bigl(\frac{-1}{p}\Bigr) \frac{-\tm_5(p)}p+\Bigl(\frac{-15}p\Bigr)p^2
	\quad
	\text{for primes $7\leq p \leq 2309$},
	$$
	and this confirms the claimed modularity for $\tm_5$.

\begin{rem}\label{rem:Evans5}
Suppose $f$ satisfies the equality \eqref{eq:Evans5}
in Theorem \ref{thm:EC}.
Let $g=f\otimes\psi$ be the twist of $f$ by a character $\psi$.
Then the coefficients $a_g(p) = \psi(p)a_f(p)$ satisfy
a similar equality,
like \eqref{eq:Evans5_2} below.
In this sense,
the modular form $f$ is the one with minimal level
that interprets $\tm_5(p)$ in terms of
(the symmetric squares of) modular forms.
We refer the computation of the level of a twist to \cite[Lemma 11.2.1]{Computing}.

In particular,
let $g$ be one of the forms
$f_{\text{300.3.g.e}}$ and $f_{\text{300.3.g.h}}$ in {\LMFDB},
which are the twists of $f_{\text{60.3.b.a}}$
by the characters
$\chi = \chi_{5.c}$ or $\chi_{15.e}$.
We have $\chi^2(n) = \big(\frac{5}{n}\big)$,
and hence one has
\begin{equation}\label{eq:Evans5_2}
\Big(\frac{-3}{p}\Big)a_g(p)^2
= p^2\Big[\Big(\frac{15}{p}\Big)\frac{-\tm_5(p)}{p^3} + 1\Big]
\end{equation}
for prime  $p>5$.
On the other hand,
we have
$\big(\frac{15}{p}\big)\frac{-\tm_5(p)}{p^3} = \tr(r_\ell(\Frob_p)) \in \QQ$
for $r_\ell(\Frob_p) \in \SO(V_{5,\ell})$
in \eqref{eq:trace_r_ell}.
Now choosing a basis of $(V_{5,\ell})_{\QQbar_\ell}$
and taking an isomorphism $\QQbar_\ell \cong \CC$
as abstract fields,
we regard $r_\ell(\Frob_p)$ as an element in $\SO(3,\CC)$.
Regarding traces, we may assume that $r_\ell(\Frob_p)$
is semisimple.
Then $r_\ell(\Frob_p)$ lies in a compact subgroup of $\SO(3,\CC)$
since the eigenvalues have absolute value one.
Therefore $r_\ell(\Frob_p)$ is conjugate to an element
in the compact form $\SO(3)$.
For $T\in\SO(3)$,
one has
$\tr(T) = 1+2\cos\theta \geq -1$
since every nontrivial rotation $T$ on $\RR^3$ is determined by
its axis and angle of rotation $\theta \in \RR/2\pi\ZZ$.
Consequently,
one can replace the term on the left-hand side of \eqref{eq:Evans5_2}
by $|a_g(p)|^2$,
which recovers Evans's original formulation
in \cite[p.524]{Evans-hyper}.
The positivity is consistent with the observation
that $a_n(g)\in \ZZ$ when $n\equiv 1 \bmod 6$ and $a_n(g)\in \ZZ \cdot \sqrt{-5}$ when $n\equiv 5 \bmod6$ in \LMFDB.
\end{rem}

\appendix
\section{Additional materials}
\subsection{Direct approaches}\label{sect:Lim}

Deligne \cite{Deligne} and Katz \cite{Katz88} had given cohomological interpretation of the Kloosterman sums. Let $p,\ell$ be different prime numbers and $\QQbar_\ell$ be an algebraic closure of the field of $\ell$-adic numbers $\QQ_\ell$.
Then the nontrivial additive character $\psi\colon \FF_p\to\CC^{\times}$ could be regard as character taking value in $\QQbar_\ell$ by choosing a primitive $p$-th root of unity in $\QQbar_\ell$. Denote the Artin-Schreier sheaf as $\AS$, which is a rank one $\ell$-adic lisse \'{e}tale sheaf on $\Ap$ with trace function $z\mapsto\psi(\tr_{\FF_q/\FF_p}(z))$. Now, consider the following diagram:

\begin{equation}\label{prod-sum}
    \vcenter{
    \xymatrix{
    &{\mathbb{G}_{m,\FF_p}^2}\ar[dl]_{\gamma} \ar[dr]^{\sigma}&\\
    {\mathbb{G}_{m,\FF_p}}&&
    {\AA_{\FF_p}^1}
    }
    }
\end{equation}

\noindent where $\sigma$ and $\gamma$ are sum and product of coordinates respectively, i.e., if the coordinates of $\Gptwo$ are $(z,x)$, then $\sigma(z,x)=z+x$ and $\gamma(z,x)=zx$. The Kloosterman sheaf $\Kl_2$ is defined as:
\begin{equation*}
    \Kl_2=\mathrm{R}^1\gamma_{!}\sigma^{*}\AS.
\end{equation*}

The Kloosterman sheaf is a lisse etale sheaf of rank two on $\Gp$ with trace function $a\mapsto -\Kl_2(q;a)$ if $a\in\mathbb{G}_m(\FF_q)$. Moreover, $\Kl_2$ is tame at $0$ and totally wild at $\infty$.

By changing variables of $\Gptwo$ by $(z,x)\mapsto(z/x,x)$, the diagram \eqref{prod-sum} could be regard as:
\begin{equation*}
    \vcenter{
    \xymatrix{
    &{\mathbb{G}_{m,\FF_p}^2}\ar[dl]_{\pi} \ar[dr]^{f}&\\
    {\mathbb{G}_{m,\FF_p}}&&
    {\AA_{\FF_p}^1}
    }
    }
\end{equation*}
where $\pi(z,x)=z, f(z,x)=x+z/x$ and $\Kl_2\simeq  \mathrm{R}^1\pi_{!}f^{*}\AS$.

Let $[2]\colon\Gp\to\Gp$, $t \mapsto t^2$ be the double cover, and $\Klt_2=[2]^{*}\Kl_2$. Furthermore, denote $\left(\frac{\cdot}{p}\right)$ as the sheaf of Legendre's symbol on $\Gp$.

First, we recall some vanishing results.

\begin{prop}
The following cohomology groups vanish.
\begin{enumerate}
    \item $\mathrm{H}^i_{\et,?}(\Gpb,\Sym^k\Kl_2)=0$ for all $?=\emptyset, c$ and $i\neq 1$.
    \item $\mathrm{H}^i_{\et,?}(\Gpb,(\frac{\cdot}{p})\otimes\mathrm{Sym^k Kl_2})=0$ for all $?=\emptyset, c$ and $i\neq 1$.
\end{enumerate}
\end{prop}

\begin{proof}
We refer the proof to \cite[Chapter 5]{Katz88}. The sheaf of Legendre symbol $\legen$ is tame at both $0$ and $\infty$; the result follows from the Grothendieck-Shafarevich-Ogg formula.
\end{proof}

Suppose $(z,x_i)$ are coordinates of $\Gpkone$.
Let $f_k=\sum_{i=1}^{k}(x_i+z/x_i)$ and $\tilde{f}_k=\sum_{i=1}^{k}(x_i+z^2/x_i)$. Denote by $\mathscr{L}_f$ the pullback of the Artin-Schreier sheaf $f^{*}\AS$. We have the commutative diagram:

\begin{equation}
    \vcenter{
    \xymatrix{
    &{\mathbb{G}_{m,\FF_p}^{k+1}}\ar[dl]_{\pi} \ar[dr]^{f_k}&\\
    {\mathbb{G}_{m,\FF_p}}&&
    {\AA_{\FF_p}^1}
    }
    }
\end{equation}

\begin{prop}
The sheaf $\Kl_2^{\otimes k}$ is isomorphic to $\mathrm{R}^k\pi_{!}\ASk$.
\end{prop}

\begin{proof}
Consider the following Cartesian diagram:
\begin{equation}
    \vcenter{
    \xymatrix{
    {\Gpkone}\ar[r]^{h}\ar[d]_{\pi}&
    {\mathbb{G}_{m,\FF_p}^{2k}}\ar[d]^{\pi_k}\ar[dr]^{f_{k}}&
    \\
    {\Gp}\ar[r]_{\Delta}&{\Gpk}&{\Ap}
    }
    }
\end{equation}
where 
\begin{align*}
   &h(z,x_i)=(z,x_1,\dots,z,x_k)&
   &\pi(z,x_i)=z\\ 
   &\pi_k(y_1,x_1,\dots,y_k,x_k)=(y_1,\dots,y_k)& &f_k(y_1,x_1,\dots,y_k,x_k)=\sum_{i=1}^{k}(x_i+y_i/x_i)
\end{align*} 
and $\Delta$ is the diagonal embedding.

Note that $\Kl_2^{\otimes k}\simeq\Delta^{*}(\Kl_2^{\boxtimes k})$. The proposition follows from applying the spectral sequence and the base change theorem.
\end{proof}

Let $\chi\colon\mathfrak{S}_k,\mathfrak{S}_k\times\mu_2\to\{\pm 1\}$ be the sign character of permutation group $\mathfrak{S}_k$ and $\chi'\colon\mathfrak{S}_k\times\mu_2\to\{\pm 1\}$ be the sign character of permutation group and nontrivial action on $\mu_2$. For a representation $V$ of the finite group $G=\mathfrak{S}_k$ or $\mathfrak{S}_k\times\mu_2$ over a field $K$ of characteristic zero, we denote $V^{G,\chi}$ or $V^{G,\chi'}$ as their isotypic parts, i.e., the image of 
\begin{gather*}
    \left(\frac{1}{|G|}\sum_{\sigma\in G}\chi(\sigma)\sigma \right) \quad \text{or} \quad \left(\frac{1}{|G|}\sum_{\sigma\in G}\chi'(\sigma)\sigma \right)
\end{gather*}
in the group ring $K[G]$ acting on $V$.

\begin{prop}
There are following isomorphisms.
\begin{enumerate}
    \item $\coH^1_{\et,\rc}(\Gpb,\Sym^k\Kl_2)\simeq \coH^{k+1}_{\et,\rc}(\Gpbkone,\ASk)^{\symgp_k,\chi}
    \simeq\coH^{k+1}_{\et,\rc}(\Gpbkone,\ASkt)^{\symgp_k\times\mu_2,\chi}$.
    \item $\mathrm{H}^1_{\et,c}(\Gpb,(\frac{\cdot}{p})\otimes\mathrm{Sym^k Kl_2})\simeq\mathrm{H}^{k+1}_{\et,c}(\Gpbkone,\ASkt)^{\mathfrak{S}_k\times\mu_2,\chi^{\prime}}$.
\end{enumerate}
\end{prop}

\begin{proof}
We refer the proof to \cite[Section 2]{FSY}. For the second statement, we note that 
\begin{gather*}
    \coH^{k+1}_{\et,c}(\Gpbkone,\ASkt) \simeq \coH^1_{\et,c}(\Gpb,\Sym^k\Kl_2) \oplus \coH^1_{\et,c}(\Gpb,(\frac{\cdot}{p}) \otimes \Sym^k\Kl_2).
\end{gather*}
Then the second statement follows from the first one.
\end{proof}

Lei Fu and Daqing Wan had studied deeply the L-functions for symmetric products of the Kloosterman sums in \cite{FW05,FW08,FW10}. We follow the ideas of Fu and Wan in \cite{FW05} to do the computations for twisted moments of the Kloosterman sums, especially the local monodromy at $\infty$.

Recall that the Kloosterman sheaf is a lisse etale sheaf of rank two on $\Gp$ with trace function $a\mapsto -\Kl_2(q;a)$ if $a\in\mathbb{G}_m(\FF_q)$. Moreover, $\Kl_2$ is tame at $0$ and totally wild at $\infty$. Lisse etale sheaves on $\Gp$ correspond to the Galois representations of the function field $\FF_p (t)$ of $\Gp$. By abuse of notations, we still denote the Galois representations corresponding to $\Kl_2,\Sym^k\Kl_2,(\frac{\cdot}{p})\otimes\mathrm{Sym^k Kl_2}$ by $\Kl_2,\Sym^k\Kl_2,(\frac{\cdot}{p})\otimes\mathrm{Sym^k Kl_2}$. Besides, we denote the decomposition group, inertia group and wild inertia group at prime $\frak{p}$ by $D_{\frak{p}},I_{\frak{p}},P_{\frak{p}}$, respectively.    

First, we summarise the results of Fu-Wan in \cite{FW05}:
\begin{itemize}
    \item Dimension of the $I_0$ invariant part of $\Sym^k\Kl_2$ is $1$.
    \item Let $F_0$ be the geometric Frobenius at $0$. Then $\det(1-F_{0}T\mid(\Sym^k\Kl_2)^{I_0})=1-T$.
    \item Dimension of the $I_{\infty}$ invariant part of $\Sym^k\Kl_2$ is given as follow:
    \begin{equation*}
        \dim (\Sym^k\Kl_2)^{I_{\infty}}=
        \begin{cases}
            0& \text{if $k$ is odd},\\
            1+[\frac{k}{2p}]& \text{if $k\equiv 0 \bmod 4$},\\
            [\frac{k}{2p}]& \text{if $k\equiv 2 \bmod 4$}.
        \end{cases}
    \end{equation*}
    \item Let $F_{\infty}$ be the geometric Frobenius at $\infty$. Then
    \begin{equation*}
        \det(1-F_{\infty}T\mid(\Sym^k\Kl_2)^{I_{\infty}})=
        \begin{cases}
            1& \text{if $2\nmid k$}\\
            (1-p^{\frac{k}{2}}T)^{m_k}& \text{if $2|k$ and $p\equiv 1 \bmod 4$}\\
            (1+p^{\frac{k}{2}}T)^{n_k}(1-p^{\frac{k}{2}}T)^{m_k-n_k}& \text{if $2|k$ and $p\equiv -1 \bmod 4$}
        \end{cases}
    \end{equation*}
    where
    \begin{equation*}
        m_k=
        \begin{cases}
           1+[\frac{k}{2p}]& \text{if $k\equiv 0 \bmod 4$}\\
            [\frac{k}{2p}]& \text{if $k\equiv 2 \bmod 4$} 
        \end{cases}
    \end{equation*}
    and $n_k=[\frac{k}{4p}+\frac{1}{2}]$.
\end{itemize}

Now, we start to prove the analogous results for twisted moments of the Kloosterman sums.

\begin{prop}\label{dim0}
For $I_0$ invariant part, we have $\mathrm{dim} ((\frac{\cdot}{p})\otimes\mathrm{Sym^k Kl_2})^{I_0}=0$.
\end{prop}

\begin{proof}
By \cite{Deligne,Katz88},
the monodromy of $\Kl_2$ at $0$ is unipotent with a single Jordan block.
Let $\gamma$ be a topological generator of the tame quotient of $I_0$.
Then there exists a basis $\{v_1,v_2\}$ of $\Kl_2$
such that $\gamma v_1=v_1$ and $\gamma v_2=v_1+v_2$.
Let $w_i=v_1^{k-i}v_2^{i}\in \Sym^k\Kl_2$.
Then $\{w_i\}_{i=0}^k$ forms a basis of $\Sym^k\Kl_2$.

Suppose
$w=\sum_{i=0}^k a_i w_i \in ((\frac{\cdot}{p})\otimes\Sym^k\Kl_2)^{I_0}$.
Recall that $(\frac{\cdot}{p})$ is tamely ramified at $0$ with
$\gamma|_{(\frac{\cdot}{p})}=-1$.
Therefore, we have
\begin{gather*}
    -a_0 v_1^k -a_1 v_1^{k-1}(v_1+v_2) -\dots-a_k(v_1+v_2)^k=a_0 v_1^k + a_1 v_1^{k-1}v_2 +\dots+a_k v_2^k .
\end{gather*}
By comparing the coefficients of
$v_2^k,v_1v_2^{k-1},\cdots, v_1^k$ successively,
we deduce that all $a_i$ vanish.
\end{proof}

\begin{cor}
Let $F_0$ be the geometric Frobenius at $0$. Then
\begin{gather*}
    \mathrm{det}(1-F_{0}T|((\frac{\cdot}{p})\otimes\mathrm{Sym^k Kl_2})^{I_0})=1.
\end{gather*}
\end{cor}
\begin{proof}
This follows straightforwardly from Proposition \ref{dim0}.
\end{proof}

Next, we investigate the local behaviour of the representations at $\infty$. We will use the same notations as Fu and Wan, and we state those terminologies first:

Let $p$ be an odd prime. Define 
\begin{gather*}
    S_k(2,p)=\{(j_0,j_1)\in \ZZ_{\geq 0}^2\;|\;j_0-j_1\equiv 0 \bmod p \;\text{and}\;  j_0+j_1=k\}.
\end{gather*}
Let $\sigma$ be the switching operator acting on $S_k(2,p)$ by $\sigma(j_0,j_1)=(j_1,j_0)$.

Suppose $V$ is a vector space with basis $\{e_0,e_1\}$. Denote 
\begin{gather*}
    v_j=e_0^j e_1^{k-j}+(-1)^{k-j} e_0^{k-j} e_1^j 
\end{gather*}as elements in $\mathrm{Sym}^k V$.

When $k$ is even, let $b_k(2,p)$ be the number of $\sigma$-orbits in $S_k(2,p)$ such that the subspace spanned by that orbit is not zero. More explicitly, 
\begin{equation*}
    b_k(2,p)=
    \begin{cases}
       1+[\frac{k}{2p}]& \text{if $k\equiv 0 \bmod 4$},\\
       [\frac{k}{2p}]& \text{if $k\equiv 2 \bmod 4$}.
    \end{cases}
\end{equation*}

\begin{prop}\label{diminf}
Dimension of the $I_{\infty}$ invariant part of $(\frac{\cdot}{p})\otimes\mathrm{Sym^k Kl_2}$ is given by the following formula:
\begin{equation*}
    \mathrm{dim}((\frac{\cdot}{p})\otimes\mathrm{Sym^k Kl_2})^{I_{\infty}}=
    \begin{cases}
       0& \text{if $k$ is odd},\\
       [\frac{k}{2p}]& \text{if $k\equiv 0 \bmod 4$},\\
       1+[\frac{k}{2p}]& \text{if $k\equiv 2 \bmod 4$}.
    \end{cases}
\end{equation*}
\end{prop}

\begin{proof}
We use the same notations as Fu-Wan in \cite[Lemma 2.1]{FW05}.
There exists a basis $\{e_0,e_1\}$ of $\Kl_2$
such that $g e_0=e_1$ and $g e_1=-e_0$. Besides, we also have $g_{(a,\mu)}(e_i)=\psi_2((-1)^{i+1}a)\chi(\mu^{-1})e_i$ where $(a,\mu)\in \FF_p \times \mu_2$.

Suppose $v=\sum_{j=0}^k a_j e_0^j e_1^{k-j}$ is invariant under the actions of $I_{\infty}$.
Then we have equation $gv=v$, that is:
\begin{gather*}
    -(\sum_{j=0}^k (-1)^{k-j}a_j e_1^j e_0^{k-j})=\sum_{j=0}^k a_j e_0^j e_1^{k-j}.
\end{gather*}
(The minus sign on the left hand side is due to the twist of the Legendre symbol.)
By comparing the coefficients of both sides, we deduce that
\begin{equation*}
    \begin{cases}
       a_j=(-1)^{j+1} a_{k-j},\\
       a_{k-j}=(-1)^{k-j+1} a_j.
    \end{cases}
\end{equation*}
Therefore, we have $a_j=(-1)^{k+2}a_j$. If $k$ is an odd number, then $a_j=0$ for all $j$ and hence $v=0$. If $k$ is an even number, then $v$ is generated by $\tilde{v}_j=e_0^j e_1^{k-j}+(-1)^{k-j+1} e_0^{k-j} e_1^j$ with $0\leq j\leq \frac{k}{2}$.

Moreover, we also have equations $g_{(a,1)}(v)=v$ for all $a\in \mathbb{F}_p$, which are:
\begin{gather*}
    \sum_{j=0}^k a_j \psi_2(-ja+(k-j)a)e_0^j e_1^{k-j}=\sum_{j=0}^k a_j e_0^j e_1^{k-j}.
\end{gather*}
Therefore, we deduce that $j-(k-j)\equiv 0 \bmod p$. Hence, we conclude that 
\begin{equation*}
\mathrm{dim}((\frac{\cdot}{p})\otimes\mathrm{Sym^k Kl_2})^{I_{\infty}}=
\begin{cases}
   0& \text{if $k$ is odd},\\
   b_k(2,p)& \text{if $k$ is even}.
\end{cases}    
\end{equation*}

Finally, we observe that $\tilde{v}_{\frac{k}{2}}$ is zero if and only if $4|k$. This completes the proof.
\end{proof}

\begin{prop}
Let $F_{\infty}$ be the geometric Frobenius at $\infty$.
Then
    \begin{equation*}
        \det(1-F_{\infty}T\mid((\frac{\cdot}{p})\otimes\Sym^k\Kl_2)^{I_{\infty}})=
        \begin{cases}
            1& \text{if $2\nmid k$}\\
            (1-p^{\frac{k}{2}}T)^{m_k}& \text{if $2|k$ and $p\equiv 1\pod 4$}\\
            (1+p^{\frac{k}{2}}T)^{n_k}(1-p^{\frac{k}{2}}T)^{m_k-n_k}& \text{if $2|k$ and $p\equiv -1 \pod 4$}
        \end{cases}
    \end{equation*}
    where
    \begin{equation*}
        m_k=
        \begin{cases}
           [\frac{k}{2p}]& \text{if $k\equiv 0 \bmod 4$}\\
        1+[\frac{k}{2p}]& \text{if $k\equiv 2\bmod 4$} 
        \end{cases}
    \end{equation*}
    and $n_k=[\frac{k}{4p}+\frac{1}{2}]$.
\end{prop}

\begin{proof}
This is an analog of the result of Fu-Wan in \cite[Lemma 4.1]{FW05}.
The proof only requires a similar modification as that of Proposition \ref{diminf}.
\end{proof}

\begin{thm}
For any odd prime $p>k$,
\begin{equation*}
    \mathrm{dim}(\mathrm{H}^1_{\et,\rmid}(\Gpb,(\frac{\cdot}{p})\otimes\mathrm{Sym^k Kl_2}))=
    \begin{cases}
       \frac{k+1}{2}& \text{if $k$ is odd},\\
       \frac{k}{2}& \text{if $k\equiv 0 \bmod 4$},\\
       \frac{k}{2}-1& \text{if $k\equiv 2 \bmod 4$}.
    \end{cases}
\end{equation*}
\end{thm}

\begin{proof}
For $p>k$, we may compute the dimension by
\begin{align*}
    \dim\coH^1_{\et,\rmid}(\Gpb,(\frac{\cdot}{p})\otimes\Sym^k\Kl_2)&=
    \dim\coH^1_{\et,\rc}(\Gpb,(\frac{\cdot}{p})\otimes\Sym^k\Kl_2) \\
    &\hspace{15pt}-\dim((\frac{\cdot}{p})\otimes\Sym^k\Kl_2)^{I_{0}}
    -\dim((\frac{\cdot}{p})\otimes\Sym^k\Kl_2)^{I_{\infty}}.
\end{align*}
Hence, we have
\begin{equation*}
    \dim(\coH^1_{\et,\rmid}(\Gpb,(\frac{\cdot}{p})\otimes\Sym^k\Kl_2))=
    \begin{cases}
       \frac{k+1}{2}-0-0& \text{if $k$ is odd},\\
       \frac{k}{2}-0-[\frac{k}{2p}]& \text{if $k\equiv 0 \bmod 4$},\\
       \frac{k}{2}-0-(1+[\frac{k}{2p}])& \text{if $k\equiv 2\bmod 4$}.
    \end{cases}
\end{equation*}
Note that $[\frac{k}{2p}]=0$ when $p>k$, and this completes the proof.
\end{proof}

\subsection{Kloosterman sums and hypergeometric character sums} \label{sect:HCS}

Over the past few decades, progress has been made in understanding the relationship between Kloosterman sums and hypergeometric functions over finite fields, see \cite{Katz88, Greene, Evans-hyper, Lin-Tu, KAS, Saikia, Otsubo} for example. The modularity of the Kloosterman representations can also be obtained from the modularity of the corresponding hypergeometric representations.  In this subsection, we summarize some of the identities related to Evans's conjectures. 

For simplicity, we fix  $\psi\colon\FF_q\to\CC^{\times}$ to be the canonical additive character of $\FF_q$ and extend the multiplicative characters on $\FF_q^\times$ to $\FF_q$ by setting $\chi(0)=0$. Denote the $g(A)$ and $J(A,B)$ the Gauss sum and Jacobi sum:
$$
 g(A):=\sum_{a\in \FF_q}A(a)\psi(a), \quad J(A,B)= \sum_{a\in \FF_q} A(a)B(1-a),
$$
where $A$, $B\in \widehat{\FF_q^\times}$. 

Built on the orthogonality relations and the fact that  
$$
  \psi(a)=\frac 1{q-1}\sum_{\chi \in \widehat{\FF_q^\times}}g(\ol \chi)\chi(a),  \quad a\in \FF_q^\times,
$$
we can write $\spm_k(q)$ and $\tm_k(q)$  in terms of Gauss and Jacobi sums.  
\begin{lem}\label{lemma: sym2}
	For $a\in \FF_q^\times$,
	$$
\Kl_2(q;a) = \frac 1{q-1}\sum_{\chi \in \widehat{\FF_q^\times}}g(\ol \chi)^2\chi(a), 
$$
$$
   \Kl_2^{\Sym^2}(a;q) =\Kl_2(q;a)^2-q = \frac 1{q-1}\frac1{g(\phi)}\sum_{\chi \in \widehat{\FF_q^\times}}g(\ol \chi)^3g(\phi\chi)\chi(-4a),
$$
where $\phi$ is the quadratic character of $\widehat{\FF_q^\times}$. 
\end{lem}
\begin{cor}
	For an odd prime power $q$, let $\phi \in \widehat{\FF_q^\times}$ be the quadratic character. Then
	\begin{align*}
		\tm_1(q) &=-g(\phi)^2=-\phi(-1)q, \qquad 		\tm_2(q) =-\phi(-1)g(\phi)^2=-q.
	\end{align*} 
\end{cor}

\medskip

Further, we define hypergeometric functions over finite fields as follows.  Let $A$ be a multiplicative character and  $t\in \FF_q^\times$,  we define 
$$
{}_{1}\mathbb{P}_{0}[A;t;q] := \ol{A}(1-t).
$$
When $n \geq 2$, for any $A_i$,  $B_j \in \widehat{\FF_q^\times}$, $i=0$, $\ldots$, $n-1$, $j=1$,  $\ldots$, $n-1$, we define 

\begin{align*}
\pPPq n{n-1}{A_0&A_1&\cdots& A_{n-1}}{&B_1&\cdots&B_{n-1}}{t;q}:=&\\
 \sum_{x\in \FF_q} A_{n-1}(x)\overline A_{n-1} B_{n-1}(1-x) \cdot
	&\pPPq{n-1}{n-2}{A_0& A_1&\cdots &A_{n-2}}{ & B_1&\cdots &B_{n-2}}{tx;q}\\
	= \frac{(-1)^{n}}{q-1} \left(\prod_{i=1}^{n-1}A_{i}B_{i}(-1) \right)& \sum_{\chi \in \widehat{\FF_q^\times}} \binom{A_{0}\chi}{\chi} \binom{A_{1}\chi}{B_{1}\chi} \cdots \binom{A_{n-1}\chi}{B_{n-1}\chi}\chi(t),
\end{align*} 
where 
$\binom{A}{B} := -B(-1)J(A, \overline{B})$,
for any characters $A,B$.   There are different versions of hypergeometric type character sums documented in literature.  
Here we use the version given in \cite{WIN3X}, for the purpose that these are closely related to the hypergeometric sheaves introduced by Katz \cite{Katz88, Katz96, LLT2}.  

	From the inductive formula 
	\begin{align}\label{eqn:KLsym^k} 
		  \Kl_2^{\Sym^n}(a;q)=  -\Kl_2(a;q) \Kl_2^{\Sym^{n-1}}(a;q)-q \Kl_2^{\Sym^{n-2}}(a;q)
	\end{align}
	and Lemma \ref{lemma: sym2}, one can induce the third and fourth moments as $\mathbb P$-values (see \cite{Lin-Tu, KAS, Saikia} for example), which we list the relation below. We will drop the $q$-notation (in our case $q=p$) in the $\mathbb P$-function.  
	
	\begin{lem}\label{lemma: reduction}
		For $a\in \FF_p^\times$,
		$$
		-\sum_{a\in \FF_p^\times}\phi(a)\Kl_2(a;p) \Kl_2^{\Sym^{2}(a;p)} =p\phi(-1)\pPPq32{\phi&\phi&\phi}{&\varepsilon&\varepsilon}{4}=p\pPPq32{\phi&\phi&\phi}{&\varepsilon&\varepsilon}{\frac14},
		$$
			$$
		\sum_{a\in \FF_p^\times}\phi(a)\(\Kl_2^{\Sym^{2}(a;p)}\)^2 =p\phi(-1)\pPPq43{\phi&\phi&\phi&\phi}{&\varepsilon&\varepsilon&\varepsilon}{1}.
		$$
		where $\phi$ is the quadratic character of $\widehat{\FF_p^\times}$. 
	\end{lem}

\begin{thm}{\cite{Ahlgren-Ono-CalabiYau, AOP, Evans-hyper}}
	  Let $p>3$ be an odd prime, and $\phi$, $\varepsilon$ be the quadratic character and trivial character of $\FF_p^\times$, respectively. Then  
	 \begin{align*}
		\tm_3(p)=&  -p\pPPq32{\phi&\phi&\phi}{&\varepsilon&\varepsilon}{\frac14}+\phi(-1)p^2=-p\cdot a_p(f_{\text{48.3.e.a}}),\\
		\tm_4(p)=&  \phi(-1)p\cdot \pPPq43{\phi&\phi&\phi&\phi}{&\varepsilon&\varepsilon&\varepsilon}{1}+p^2=-p\cdot a_p(f_{\text{8.4.a.a}}),
	\end{align*}
	  where $a_p(f)$ is the $p$th Fourier coefficient of the Hecke eigenform $f$. 
\end{thm}

\medskip

For higher powers $k$, the expressions involve inductive formula of $\mathbb P$-functions. 
	 \begin{prop}
	 	  Let $p>3$ be an odd prime, and $\phi$, $\varepsilon$ be quadratic character and trivial character of $\FF_p^\times$, respectively. We have 
	 	  \begin{align*}
	 		\tm_5(p) 
	 		=&   -p\sum_{t \atop{y\neq 0, 1, 1/4t}} \phi\(\frac t{t-1}\) \pPPq32{\phi&\phi&\phi}{&\varepsilon&\varepsilon}{(1-y)(t-\frac1{4y})}\\
	 		& +2p^2\pPPq32{\phi&\phi&\phi}{&\varepsilon&\varepsilon}{\frac14} +2p+p(p+1)\(\frac{-3}p\).
	 	\end{align*}
	 \end{prop}
	 \medskip

	To express $\tm_5$ in terms of $\mathbb P$-functions, we will apply the following basic properties. 
	\begin{lem}\label{lem: Gauss-Job}
		Let $A$, $B\in\widehat{\FF_q^\times}$. 
	   \begin{align*}
		J(A,B) =& \frac{g(A)g(B)}{g(AB)}, \quad AB\neq \varepsilon, \\
	     g(A)g(\ol A) &=A(-1)q, \quad A\neq \varepsilon,
		\end{align*}
		and, for $a\in \FF_q^\times$, 
		$$
		  \frac1{q-1}\sum_{K\in \widehat{\FF_q^{\times}}}J(AK,\ol K)K(a) = A(1-a). 
		$$
		
	\end{lem}

\begin{proof}
	From the inductive formula \eqref{eqn:KLsym^k}, for a fixed prime $p$, we can simplify the term $	\Kl_2^{\Sym^5}(a;p)$  as
	 \begin{align*}
	\Kl_2^{\Sym^5}(a;p) =& \Kl_2^{\Sym^1}(a;p) \Kl_2^{\Sym^4(a;p)}-p\cdot \Kl_2^{\Sym^{3}}(a;p)\\
		=&\Kl_2^{\Sym^1}(a;p) \(\Kl_2^{\Sym^2}(a;p)\)^2 -2p\Kl_2^{\Sym^1}(a;p) \Kl_2^{\Sym^2}(a;p),
	\end{align*}
	and hence,
	  \begin{align*}
		\tm_5(p) 
		=&\sum_{a\in \FF_q}\phi(a)\Kl_2^{\Sym^1}(a;p) \(\Kl_2^{\Sym^2}(a;p)\)^2 -2p\sum_{a\in \FF_q}\phi(a)\Kl_2^{\Sym^1}(a;p) \Kl_2^{\Sym^2}(a;p).
	\end{align*}
	For short hand, we write $\tm_5(p) =I_1+I_2$, where $I_1$ and $I_2$ are the first and second terms in the above identity, respectively.   Note the $I_2$ is related to $\tm_3(p)$ and is 
	$$
	  -I_2= 2p^2\pPPq32{\phi&\phi&\phi}{&\varepsilon&\varepsilon}{\frac14},
	$$
	which we will omit the detail here. 
	
	For $I_1$, by Lemma \ref{lemma: sym2} and orthogonality property, we deduce that
	\begin{align*}
	  -I_1 =&\frac1{p(p-1)^2} \sum_{\chi \in \widehat{\FF_p^\times}} g(\ol \chi)^3g(\phi\chi) \sum_{K\in \widehat{\FF_p^\times}}g(\ol K)^2 g(\ol{K\chi})g(\phi\chi K)^3\ol K(-4)\\
	      {\overset{\chi\mapsto \phi\ol \chi}=}& \frac1{p(p-1)^2} \sum_{\chi \in \widehat{\FF_p^\times}} g(\phi \chi)^3g(\ol \chi) \sum_{K\in \widehat{\FF_p^\times}}g(\ol K)^2 g(\phi\chi\ol{K})g(\ol\chi K)^3\ol K(-4).
	\end{align*}
 We now apply the basic properties in Lemma \ref{lem: Gauss-Job} to $I_1$. 
 Firstly, we split $I_1$ into two parts: 
		\begin{align*}
		I_1 =&\frac{g(\phi)}{p(p-1)^2} \sum_{\chi \neq \varepsilon}  g(\phi \chi)^3g(\ol \chi)^3 \sum_{K\in \widehat{\FF_p^\times}}J(\ol \chi K, \ol K)^2 J(\ol \chi K, \phi\chi\ol K) \ol K(-4)\\
		&-\frac{g(\phi)^3}{p(p-1)^2}  \sum_{K\in \widehat{\FF_p^\times}} g(\ol K)^2g(\phi\ol K)g (K)^3\ol K(-4)\\
		   =&I_{1,1}+I_{1,2},
	\end{align*}
 where
		\begin{align*}
		I_{1,1} :=&\frac{g(\phi)}{p(p-1)^2} \sum_{\chi \in \widehat{\FF_p^\times}}  g(\phi \chi)^3g(\ol \chi)^3 \sum_{K\in \widehat{\FF_p^\times}}J(\ol \chi \ol K,  K)^2 J(\ol \chi \ol K, \phi\chi K)  K(-4),\\
		I_{1,2} :=&\frac{g(\phi)^4}{p(p-1)^2}\sum_{K\in \widehat{\FF_p^\times}}J( K, \ol K)^2 J( \ol K, \phi K) \ol K(-4)\\
		         &-\frac{g(\phi)^3}{p(p-1)^2}  \sum_{K\in \widehat{\FF_p^\times}} g(\ol K)^2g(\phi\ol K)g (K)^3K(-4).
	\end{align*}
	By Lemma \ref{lem: Gauss-Job}, we can further simplify $I_{1,2}$ as 
	\begin{align*}
	I_{1,2} =&\frac{g(\phi)^4}{p(p-1)^2}\(\sum_{K\in \widehat{\FF_p^\times}} J(\ol  K, \phi K) \ol K(-4)-J(\varepsilon,\phi)-J(\varepsilon,\varepsilon)^2\)\\
	&-\frac{g(\phi)^4}{p(p-1)^2}  \(p^2\cdot \sum_{K\in \widehat{\FF_p^\times}} J( K, \phi\ol K) K(-4) -p^2\cdot J(\varepsilon,\phi) -1\)\\
	=&\frac{p^2}{p(p-1)^2}\(\phi(-3)(p-1)+1-(p-2)^2-\phi(-3)p^2(p-1)-p^2+1\)\\
	=&-\phi(-3)p(p+1)-2p.
\end{align*}
Regarding $I_{1,2}$, we will make a change of variable using $K\mapsto \ol \chi K$ to make the expression simpler. That is,
	\begin{align*}
	I_{1,1} =&\frac{g(\phi)}{p(p-1)^2} \sum_{\chi \in \widehat{\FF_p^\times}}  g(\phi \chi)^3g(\ol \chi)^3 \chi(-4)\sum_{K\in \widehat{\FF_p^\times}}J(\ol \chi  K, \ol  K)^2 J(\phi K, \ol K)  K(-4),\\
	&=-\frac{g(\phi)^4}{p(p-1)^2} \sum_{\chi \in \widehat{\FF_p^\times}}  J(\phi \chi, \ol \chi)^3 \ol \chi(4)\cdot (p-1)(-\phi(-1)) \pPPq32{\ol \chi&\ol \chi&\phi}{&\varepsilon&\varepsilon}4\\
		&=\frac{\phi(-1)p}{p-1} \sum_{\chi \in \widehat{\FF_p^\times}}  \CCC{\phi\chi}{\chi}^3\chi(-1/4)\cdot \pPPq32{\ol \chi&\ol \chi&\phi}{&\varepsilon&\varepsilon}4.
\end{align*}
From the $\mathbb P$-expression, we can write 
$$
	 \pPPq32{\ol \chi&\ol \chi&\phi}{&\varepsilon&\varepsilon}4 = \sum_{t,y} \phi\(\frac{t}{1-t}\)\ol \chi(y)\chi(1-y)\chi(1-4ty).
$$
This gives  
	\begin{align*}
	I_{1,1}
	&=\frac{\phi(-1)p}{p-1} \sum_{t,y} \phi\(\frac{t}{1-t}\) \sum_{\chi \in \widehat{\FF_p^\times}}  \CCC{\phi\chi}{\chi}^3\chi\(\frac{-(1-y)(1-4ty)}{4y}\)\\
	&=\phi(-1) \sum_{t,y} \phi\(\frac{t}{1-t}\)  \pPPq32{\phi&\phi&\phi}{&\varepsilon&\varepsilon}{(1-y)(t-\frac1{4y})},
\end{align*}
which leads to the desired expression in the Proposition.
\end{proof}

	 \begin{rem}    As an application in the theory of modular forms, it is worthy to mention that it would be interesting to investigate relation between the double sums involved hypergeometric functions in the proposition and certain space of cusp forms.  In the classical case, the space of cusp forms on $\Gamma_1(N)$ are spanned by Poincare series, whose Fourier coefficients are (twisted) Kloosterman sums (see \cite[Section 9.3]{Cohen-Stromberg} for example).   In the works of  \cite{Scholl, HLLT, FOP, Fuselier, Lennon1, Lennon2}, the traces of Hecke operators on the certain subgroups of triangle groups commensurable with $\mbox{PSL}_2(\ZZ)$ can be computed via $\mathbb P$-functions and suitable rational functions as the arguments. It might be interesting to see the connection from what  the modularity of the Kloosterman sheaves suggest.  
	
	 \end{rem}

\subsection{$L$-functions}\label{sect:A.L-func}
We list what we know about the shapes
of the $L$-functions $L_k(s)$
and the completed one $\Lambda_k(s)$
associated with the Galois representations $\rho_{k,\ell}$
(or the motive $\Motive_k$)
in Theorem \ref{thm:GR},
as in \cite[\S 5]{FSY} and \cite[\S 8]{FSY:Bessel}.
Since $\rho_{k,\ell}$ is potentially automorphic,
$\Lambda_k(s)$ extends meromorphically to $\CC$
and admits the functional equation
$\Lambda_k(k+2-s) = \pm\Lambda_k(s)$.

\subsubsection*{Euler factors}
By Propositions \ref{oddrep} and \ref{evenrep},
the local Euler factor $L_k(p;p^{-s})^{-1}$
of $L_k(s)$ at an odd prime $p$ is given by
\begin{align*}
L_k(p;T) &= \det(1-F_pT\mid V_{k,\ell}^{I_p}) \\
&= \begin{cases}
Z_k(p;T) \prod_{a\in\Theta_k(p)^+}
\Big( 1- \Big(\frac{(-1)^{(1+ap)/2}2a'}{p}\Big) p^{(k+1)/2}T \Big)
& \text{for $k$ odd}, \\
Z_k(p;T) (1-p^{k/2})^{-\delta_{2+4\ZZ}(k)}
& \text{for $k$ even},
\end{cases}
\end{align*}
where
$Z_k(p;T)$ and $\Theta_k(p)^+$ are defined in
\eqref{eq:localfactor_Z} and \eqref{eq:Thetapm}, respectively, and
$a' = ap^{-v_p(a)}$ is the prime-to-$p$ part of $a$.
For $k\leq 6$,
the local Euler factor at $2$ is trivial.

\subsubsection*{Conductors}
In general, the compatible system $\{\rho_{k,\ell}\}_\ell$
is wildly ramified at prime $2$.
We write its conductor as
\[ \cond_k = 2^c\cond'_k \]
where $\cond'_k$ is odd.
We have
\[ \cond'_k = \begin{cases}
1_\boxtimes 3_\boxtimes 5_\boxtimes \cdots k_\boxtimes,
& \text{$k$ odd}, \\
2_\vartriangle 4_\vartriangle 6_\vartriangle \cdots k_\vartriangle,
& \text{$k$ even}.
\end{cases} \]
Here $n_\boxtimes$ (the square-free part)
is the product of prime divisors $p$ of $n$
with odd valuation $v_p(n)$,
and $n_\vartriangle$ (the odd part of the radical)
is the product of odd prime divisors.
For example,
for $k=1,3,5$,
we have $c=2,4,6$, respectively,
while for $k=4,6$, $c=3$.

\subsubsection*{Gamma factors}
The gamma factor of $\Motive_k$
is determined by the Hodge numbers of its de Rham realization $V_{k,\dR}$
together with the action $\iota$ induced by the complex conjugation.
Notice that $\iota$ switches
$V_{k,\dR}^{pq}$ and $V_{k,\dR}^{qp}$
in the Hodge decomposition,
and $V_{k,\dR}^{(k+1)/2,(k+1)/2} \neq 0$
only when $k\equiv 1\bmod{4}$.
In the latter, $h^{(k+1)/2,(k+1)/2} = 1$.
Consequently,
by \eqref{eq:det_V},
$\iota$ acts on $V_{k,\dR}^{(k+1)/2,(k+1)/2}$
by
\[ (-1)^{(k-1)/4}\det\rho_{k,\ell}(\iota) = 1 = -(-1)^{(k+1)/2}. \]

The gamma factor of $\Motive_k$ is then given by
\[ L_k(\infty,s) = \pi^{-ms/2}
\prod_{j=1}^m\Gamma\Big(\frac{s+1-j}{2}\Big),
\quad
m = \rk\Motive_k = \flr{\frac{k+1}{2}}-\delta_{2+4\ZZ}(k). \]

\subsubsection*{Periods}
For the calculation of the period realizations
of the twisted moments of the Bessel connection $\Kl_2$ over $\QQ$,
see \cite{Chuang}.
Following the strategy of \cite[\S 7]{FSY:Bessel},
we now show that they coincide with the period realizations
of the motives $\Motive_k$,
focusing on constructing certain rapid decay cycle classes
for $(U_0, \tilde{f}_k)$
(i.e., rapid decay for the function $e^{-\tilde{f}_k}$).
Here, $U_0 = \Gm^{k+1}$ over $\QQ$
with Cartesian coordinates $\{t,y_i\}_{i=1}^k$
and $\tilde{f}_k = \frac{t}{2}\sum_{i=1}^k y_i+y_i^{-1}$,
regarded as a regular function on $U_0$.\footnote{
One uses the change of variables
$z = (t/2)^2, x_i = (t/2)y_i$
so that $\sum_{i=1}^k x_i+zx_i^{-1} = \tilde{f}_k$.}
Consider the action of $\symgp_k\times\mu_2$ on $(U_0,\tilde{f}_k)$,
extending that on $\sK\subset\Gm^k$
in the beginning of \S\ref{sect:motives},
where $\symgp_k$ acts trivially on the variable $t$
and the generator of $\mu_2$ sends $t$ to $-t$.
One has
$\Motive_k = W_{k+1}\coH^{k+1}(U_0,\tilde{f}_k)_\chi$,
the $\chi$-isotypic part of the weight $k+1$ piece.

Let $U = U_0(\CC)$.
We consider (oriented, unbounded) chains on $U$.
Let $c_p$ be the circles $|t|=1$ if $p=0$,
and $|y_p|=1$ for $1\leq p\leq k$,
equipped with the counterclockwise orientation.
Consider the chains
\[\textstyle \tilde\alpha_i' = [1,\infty) \times\big(\prod_{p=1}^i c_p\big)
	\times \RR_{>0}^{k-i},
\quad
0\leq i\leq \flr{(k-1)/2} \]
with boundaries
\[\textstyle \partial\tilde\alpha_i' = -\{1\} \times\big(\prod_{p=1}^i c_p\big)
	\times \RR_{>0}^{k-i}. \]
For $1\leq r\leq k$,
define the chain (\cite{FSY:Bessel})
$\tilde{\gamma}_r\colon [0,1]^r\times (0,1)^{k+1-r} \to U$
by
\[
(s_0, \ldots, s_k) \mapsto
\begin{cases}
t=e^{2\pi i \cdot s_0}, & \\
y_p=e^{2\pi i \cdot s_p}, & 1\leq p< r, \\
y_p=\begin{cases}
e^{2\pi i \cdot s_0} \cdot 3s_p, & 0<s_p\leq 1/3, \\
e^{2\pi i \cdot s_0(3-6s_p)}, & 1/3\leq s_p\leq 2/3, \\
e^{-2\pi i \cdot s_0}\cdot\frac{1}{2}\frac{s_p}{1-s_p}, & 2/3\leq s_p<1,
\end{cases} & r\leq p\leq k. \\
\end{cases}\]
In particular,
$\partial\tilde\gamma_k = -2\{1\}\times\prod_{p=1}^k c_p$.
Consider the half of $\tilde\gamma_r$
\[ \tilde\xi_r(s_0,s_p) = \tilde\gamma_r(2^{-1}s_0,s_p),
\quad
\begin{aligned}
0\leq s_p \leq 1 \quad \text{if $0\leq p < r$}, \\
0< s_p < 1 \quad \text{if $r\leq p\leq k$},
\end{aligned}\]
with boundary
\[\textstyle
\partial\tilde\xi_r= \{-1\} \times (\prod_{p=1}^{r-1} c_p)
	\times \prod_{p=r}^k (\RR_{<0} -c_p)
	- \{1\} \times (\prod_{p=1}^{r-1} c_p) \times \RR_{>0}^{k-r+1}.
\]
Then the function $e^{-\tilde{f}_k}$
decays rapidly along the chains $\tilde\alpha'_i,\tilde\gamma_r,\tilde\xi_r$.
We glue these chains to form cycles.

\newcommand{\cS}{\mathcal{S}}
Let $(\vartheta_n\in\QQ)_{n\geq 0}$ be the sequence given by $\vartheta_0=1$
and the recursive relations\footnote{
In fact,
$\vartheta_n$ is the constant term of the Euler polynomial $E_n(x)$
so that $\sum_{n=0}^\infty \vartheta_n\frac{z^n}{n!} = \frac{2}{e^z+1}$
(\cite[(15)]{Chuang}).}
\begin{equation}\label{eq:euler_const}
\sum_{p=0}^{n-1}\binom{n}{p}\vartheta_p + 2\vartheta_n = 0,
\quad n\geq 1.
\end{equation}
Set $\tilde\xi_{k+1} = \tilde\gamma_k$.
For $0\leq i\leq \flr{(k-1)/2}$, let
\[ \tilde\alpha_i = \frac{1}{|\symgp_{k-i}\times\mu_2|}
\sum_{g\in\symgp_{k-i}\times\mu_2 }
\chi(g)g\Big(\tilde\alpha_i' - \frac{1}{2}
	\sum_{r=0}^{k-i}(-1)^r\binom{k-i}{r}\vartheta_r\tilde\xi_{i+r+1}\Big). \]
where we regard $\symgp_{k-i}\times\mu_2$
as a subgroup of $\symgp_k\times\mu_2$
by letting $\symgp_{k-i}$ act on the last $k-i$ components.
By \eqref{eq:euler_const},
one checks readily that $\partial\tilde\alpha_i = 0$.

\begin{prop}
The period structure of $\Motive_k$
coincides with the one obtained in \cite[Cor.26]{Chuang}.
\end{prop}

\begin{proof}
Consider the $\chi$-symmetrization of $\tilde\alpha_i$
\[ \tau_i = \frac{1}{|\symgp_k\times\mu_2|}
	\sum_{g\in\symgp_k\times\mu_2} \chi(g)g(\tilde\alpha_i),
\quad
0\leq i \leq\flr{\frac{k-1}{2}}. \]
They define elements in the $\chi$-isotypic part
of the rapid decay homology of middle degree of $(U,\tilde{f}_k)$
with rational coefficients.
Let $w_j = t^{2j}\de t\prod_{p=1}^k\de\log y_p, 0\leq j\leq \flr{(k-1)/2}$,
regarded as elements in the $\chi$-isotypic part
of the middle de Rham cohomology
of $(U_0,\tilde{f}_k)$ over $\QQ$.
Direct calculation of the periods of the cycles $\tau_i$
against the forms $w_j$ yields the assertion
(cf., the proof of \cite[Cor.7.12(2)]{FSY:Bessel}).
\end{proof}

\subsubsection*{Deligne's Conjecture}
This is about the special values of the $L$-functions.
Similar to \cite[\S 8]{FSY:Bessel},
one can compute explicitly the so-called critical values $c_n$,
up to an algebraic multiple.
The critical integers $n$ and the values $c_n$ are listed in the following
Table \ref{table:critical}.
The conjecture then says that
the value $L_k(n)$ is a rational multiple of $c_n$.

We define
\begin{align}
\label{eq:Deven}
D_{k,\reven} &= \det\begin{pmatrix}
\int_0^\infty I_0(t)^{2i}K_0(t)^{k-2i} t^{2j} \,\de t
\end{pmatrix}_{0\leq i,j\leq\flr{(k-1)/4}}, \\
\nonumber
D_{k,\rodd} &= \det\begin{pmatrix}
\int_0^\infty I_0(t)^{2i+1}K_0(t)^{k-1-2i} t^{2j} \,\de t
\end{pmatrix}_{0\leq i,j\leq\flr{(k-3)/4}}.
\end{align}
If $k\equiv 2\bmod{4}$,
we let $D_{k,\reven}'$
be the determinant of the submatrix
obtained by removing the last rows and columns of the matrix
in \eqref{eq:Deven}.

For example,
direct computations yield
\[ D_{1,\reven} = \int_0^\infty K_0(t)\,\de t = \frac{\pi}{2},
\quad
D_{1,\rodd} = 1,
\quad
D_{2,\reven} = \int_0^\infty K_0(t)^2\,\de t = \frac{\pi^2}{4}. \]
Furthermore,
let $\chi = \big(\frac{-1}{\cdot}\big), g = f_{60.3.b.a}$
as appeared in Theorem \ref{thm:EC}(iii).
We have\footnote{
For the second equalities in \eqref{eq:D_k:example}
relating Bessel moments to hypergeometric evaluations,
see \cite[(47)]{BBBG},
sequence A263490 in \cite{OEIS}
and \cite[(73)]{BBBG}, respectively.
The identifications between the latter and the critical values
of the $L$-functions of the relevant modular forms
can be found in \cite{AMNT, Zagier}.}
\begin{equation}\label{eq:D_k:example}
\begin{aligned}
D_{3,\reven} &= \int_0^\infty K_0(t)^3\,\de t
= \frac{\sqrt{3}\pi^3}{8}\pFq32{\frac{1}{2}&\frac{1}{2}&\frac{1}{2}}{&1&1}{\frac{1}{4}}
= \frac{\pi^2}{2}L(f_{48.3.e.a},1), \\
D_{3,\rodd} &= \int_0^\infty I_0(t)K_0(t)^2\,\de t
= \frac{\pi^2}{4} \pFq32{\frac{1}{2}&\frac{1}{2}&\frac{1}{2}}{&1&1}{\frac{1}{4}}
= 2L(f_{48.3.e.a},2), \\
D_{4,\reven} &= \int_0^\infty K_0(t)^4\,\de t
= \frac{\pi^4}{4}\pFq43{\frac{1}{2}&\frac{1}{2}&\frac{1}{2}&\frac12}{&1&1&1}{1}
= 4\pi^2L(f_{8.4.a.a},2),
\end{aligned}
\end{equation}
\begin{equation}\label{eq:D_k:example?}
\begin{aligned}
D_{4,\rodd} &= \int_0^\infty I_0(t)K_0(t)^3\,\de t
\overset{?}{=} 8L(f_{8.4.a.a},3), \\
D_{5,\rodd} &= \int_0^\infty I_0(t)K_0(t)^4\,\de t
\overset{?}{=} \pi^2L(\chi\cdot\Sym^2 g,2)
= 2^3\sqrt{15}L(\chi\cdot\Sym^2 g,3), \\
D_{6,\reven}' &= \int_0^\infty K_0(t)^6\,\de t
\overset{?}{=} 108\pi^2L(f_{24.6.a.a},4), \\
D_{6,\rodd} &= \int_0^\infty I_0(t)K_0(t)^5\,\de t
\overset{?}{=} 144L(f_{24.6.a.a},5).
\end{aligned}
\end{equation}
Here, the equalities with an ``?" in \eqref{eq:D_k:example?}
have been checked numerically.

When $k=1$, following from the fact that
$L(\chi,n)\sim 1$ if $n=-2a$
and $L(\chi,n)\sim \pi^{2a+1}$ if $n=2a+1$,
Deligne's conjecture holds in this case.
(The $L$-function $L(\chi,s)$ is also known as the
Dirichlet beta function $\beta(s)$.)
On the other hand,
by Theorem \ref{thm:EC},
the identities in \eqref{eq:D_k:example},
and the vanishing $L(f_{24.6.a.a},3) = 0$ of the central value
due to the negative sign of its functional equation,
the conjecture also holds for $k=3$
and $(k,n) = (4,3), (6,4)$.

\begin{table}[htb]
\renewcommand{\arraystretch}{1.2}
\begin{center}
\begin{tabular}{|c|c|c|c|}
\hline
$k$ & rank & critical integers $n$ & $c_n$ \\
\hline\rule{0pt}{3ex}
\multirow{2}{*}{$1$} & \multirow{2}{*}{$1$} & $1-2a$ & $D_{1,\rodd}=1$ \\
& & $2+2a, a\geq 0$ & $\pi^{2a} D_{1,\reven}\sim\pi^{2a+1}$
\\
\hline\rule{0pt}{3ex}
\multirow{2}{*}{$4r - 1$} & \multirow{2}{*}{$2r$}
& $2r$ & $\pi^{-r(r+1)}D_{k,\reven}$ \\
&& $2r + 1$
& $\pi^{-r(r-1)}D_{k,\rodd}$ \\
\hline\rule{0pt}{3ex}
\multirow{3}{*}{$4r$} & \multirow{3}{*}{$2r$}
& $2r$ & $\pi^{-r(r+1)}D_{k,\rodd}$ \\
&& $2r + 1$
& $\pi^{-r(r+1)}D_{k,\reven}$ \\
&& $2r + 2$
& $\pi^{-r(r-1)}D_{k,\rodd}$ \\
\hline\rule{0pt}{3ex}
\multirow{2}{*}{$4r + 1$} & \multirow{2}{*}{$2r+1$}
& $2r + 1$ & $\pi^{-r(r+1)}D_{k,\rodd}$ \\
&& $2r + 2$
& $\pi^{-r(r+1)}D_{k,\reven}$ \\
\hline\rule{0pt}{3ex}
\multirow{5}{*}{$4r+2$} & \multirow{5}{*}{$2r$}
& $2r$ & $\pi^{-r(r+3)}D_{k,\rodd}$ \\
&& $2r + 1$
& $\pi^{-r(r+3)}D'_{k,\reven}$ \\
&& $2r + 2$
& $\pi^{-r(r+1)}D_{k,\rodd}$ \\
&& $2r + 3$
& $\pi^{-r(r+1)}D'_{k,\reven}$ \\
&& $2r + 4$
& $\pi^{-r(r-1)}D_{k,\rodd}$ \\
\hline
\end{tabular}
\vspace{5pt}
\caption{\label{table:critical}Critical integers $n$ and values of $c_n$ ($r\geq 1$)}
\end{center}
\end{table}

\bibliographystyle{plain}
\bibliography{Evans_ref}

\end{document}